\documentclass[11pt]{amsart}
\usepackage[english]{babel}
\usepackage[T1]{fontenc}
\usepackage[ansinew]{inputenc}
\usepackage{amsmath}
\usepackage{amsfonts}
\usepackage{ mathrsfs }
\usepackage{amssymb}
\usepackage{textcomp}
\usepackage{enumitem}
\usepackage[hidelinks]{hyperref}
\usepackage[arrow, matrix, curve]{xy}
\usepackage{comment}
\usepackage{color}
\usepackage{verbatim}
\usepackage{tikz}
\usetikzlibrary{decorations.pathreplacing}
\usepackage{amsthm}
\usepackage{stmaryrd}

\numberwithin{paragraph}{section}
\numberwithin{equation}{section}

\newtheorem{satz}{Theorem}[section]
\newtheorem{Thm}[satz]{Theorem}
\newtheorem{lem}[satz]{Lemma}

\newtheorem{prop}[satz]{Proposition}

\newtheorem{kor}[satz]{Corollary}
\theoremstyle{definition}
\newtheorem{defn}[satz]{Definition}
\newtheorem{bem}[satz]{Remark}
\newtheorem{nurnum}[satz]{}
\newtheorem{Ex}[satz]{Example}
\newtheorem{theorem}{Theorem}[]
\newtheorem{thm}[theorem]{Theorem}
\newtheorem{koro}[theorem]{Corollary}

\newtheorem*{ack}{Acknowledgements}
\newtheorem*{term}{Terminology}

\newcommand{\N}{\mathbb{N}}
\newcommand{\Z}{\mathbb{Z}}

\newcommand{\R}{\mathbb{R}}
\newcommand{\C}{\mathbb{C}}

\newcommand{\Xan}{X^{\an}}

	\DeclareMathOperator{\an}{an}

	\DeclareMathOperator{\Hom}{Hom}
	
	\DeclareMathOperator{\Spec}{Spec}

	\DeclareMathOperator{\supp}{supp}

\DeclareMathOperator{\CS}{\mathcal{C}}

\DeclareMathOperator{\US}{\mathcal{U}}

\DeclareMathOperator{\PB}{\mathbb{P}}

\DeclareMathOperator{\HM}{\mathbb{H}}

\def\quotient#1#2{\raise0.75ex\hbox{$\,#1$}\big/\lower0.75ex\hbox{$#2\,$}}

\title[Energy Minimization Principle for non-archimedean curves]{Energy Minimization Principle for non-archimedean curves}

\author[V.~Wanner]{Veronika Wanner}
\address{V. Wanner, Mathematik, Universit{\"a}t 
Regensburg, 93040 Regensburg, Germany}
\email{veronika.wanner@mathematik.uni-regensburg.de}

\thanks{The author was supported by the collaborative research center SFB 1085 'Higher Invariants' funded by the Deutsche Forschungsgemeinschaft.
 }

\begin{document}
\begin{abstract}
Baker and Rumely defined a notion of Arakelov--Green's functions on the Berkovich analytification of the projective line and established an Energy Minimization Principle.
We extend their definition and show their Energy Minimization Principle for general smooth projective curves.
As an application we get a generalization and a different proof of an equidistribution result by Baker and Petsche.
\bigskip

\bigskip

\end{abstract}
\maketitle 
\tableofcontents

\section{Introduction}

Potential theory is a very old area of mathematics and has been extended
to non-archimedean analytic geometry by many different authors. In the one-dimensional case this is for example done by Favre and Jonsson in \cite{FJ} for the Berkovich projective line (indeed for any metric $\R$-tree), by  Thuillier in \cite{Th} for general analytic curves and
by Baker and Rumely in \cite{BR} also for the Berkovich projective line. 
One important theorem in potential theory is the so called  Energy Minimization 
Principle.
There are independent approaches of a non-archimedean version of this principle in the case of  the Berkovich projective line $\mathbb{P}^{1,\an}$, one by Favre and Rivera-Letelier in \cite{FR} and one by 
Baker and Rumely established in  \cite{BR}.
Both results are respectively used in \cite{FR} and in \cite{BR2} as key tools for non-archimedean equidistribution results.

In this paper, we generalize Baker and Rumely's approach and extend all of their needed notions to the Berkovich analytification $\Xan$ of a smooth projective curve $X$ over an algebraically closed non-archimedean field $K$. As an application we get a generalization and a different proof of an equidistribution result by Baker and Petsche in \cite{BP}. This work is part of the author's thesis \cite{Wa}.
As Baker and Rumely's non-archimedean potential theory is only established  for the Berkovich analytification $\mathbb{P}^{1,\an}$ of the projective line, we work in Thuillier's general theory which he developed in his thesis \cite{Th}.
Most important to us are his class of smooth functions $A^0$ with its corresponding measure valued Laplacian $dd^c$ and his class of subharmonic functions.

For the  Energy Minimization Principle, we need Arakelov--Green's functions $g_\mu$ defined on~$\Xan\times\Xan$ for given  probability measures $\mu$ on~$\Xan$ with continuous potentials analogous to the complex geometrical setting. For the definition of having continuous potentials we refer to Definition \ref{Dlog-c}. This condition assures $g_\mu$ to be well-defined and to be lower semi-continuous on $\Xan\times\Xan$. 
Complex  Arakelov--Green's functions are characterized by a special list of properties. We extend the construction of Arakelov--Green's function from \cite[\S 8.10]{BR} to our general smooth projective curve $X$ such that the following analogous list is satisfied:

\begin{thm}
	For a probability measure $\mu$ on~$\Xan$ with continuous potentials, there exists a unique symmetric function $g_\mu\colon \Xan\times \Xan \to (-\infty,\infty] $ such that the following holds.
	\begin{enumerate}
		\item(Semicontinuity) The function $g_\mu$  is finite and continuous off the diagonal and strongly lower semi-continuous on the diagonal in the sense that 
		$$g_\mu(x_0,x_0)= \liminf _{(x,y)\to (x_0,x_0) ,x\neq y}g_\mu(x,y).$$
		\item (Differential equation) For each fixed $y\in \Xan$ the function  $g_\mu(\cdot,y)$ satisfies 
		  $$ dd^c g_\mu(\cdot,y)=\mu - \delta_y,$$ i.e. $\int  g_\mu(x,y)~ (dd^c f)(x) = \int f~d(\mu-\delta_y)(x)$ for all $f\in A_c^0(\Xan)$.
		\item(Normalization)
		$$ \int\int g_\mu(x,y)~d\mu(x)d\mu(y)=0.$$\end{enumerate}
\end{thm}
The function $g_\mu$ is called the \emph{Arakelov--Green's function} corresponding to $\mu$. 
With the help of $g_\mu$, we can define the \emph{$\mu$-energy integral} of an arbitrary probability measure $\nu$ on~$\Xan$ as 
$$I_\mu(\nu):=\int \int  g_{\mu}(x,y) ~d\nu (y)d\nu(x).$$
 In Theorem~\ref{THM} we formulate and prove the following Energy Minimization Principle analogous to the one in complex potential theory and \cite[\S 8.10]{BR}:

\begin{thm}[Energy Minimization Principle]\label{Intro EMP}
		Let $\mu$ be a probability measure on~$\Xan$ with continuous potentials. Then 
	\begin{enumerate}
		\item $I_\mu(\nu)\geq 0$ for each probability measure $\nu$ on~$\Xan$, and 
		\item $I_\mu(\nu)=0$ if and only if $\nu=\mu$.
	\end{enumerate}
\end{thm} 
As a direct application of the Energy Minimization Principle, we can give a generalization and a different proof of the non-archimedean local discrepancy result from \cite{BP} for an elliptic curve $E$ over $K$.
Note that in \cite{BP} everything was worked out for $K$ coming from a number field. 
For our general $K$, we define the \emph{local discrepancy} of a subset $Z_n\subset E(K)$  consisting of $n$ distinct points as
$$D(Z_n):=\frac{1}{n^2}\left( \sum_{P\neq Q\in Z_n}g_{\mu_E}(P,Q) + \frac{n}{12}\log ^+|j_E|\right),$$ where $\mu_E$ is the canonical measure and $j_E$ is the $j$-invariant of $E$ (see Section \ref{SubLD} for definitions).
Note that this definition is consistent with the definition of local  discrepancy from \cite{BP} and \cite{Pe}.
We show in Corollary \ref{KorLD} the following generalization of \cite[Corollary 5.6]{BP} using the Energy Minimization Principle:
\begin{koro}
	For each $n\in \N$, let $Z_n\subset E(K)$  be a set consisting of $n$ distinct points and let $\delta_n$ be the probability measure on~$E^{\an}$ that is equidistributed on~$Z_n$. If $\lim_{n\to \infty} D(Z_n)=0$, then $\delta_n$ converges weakly to $\mu_E$ on~$E^{\an}$.
\end{koro}

\term{In this paper, let $K$ be an algebraically closed field endowed with a complete, 
non-archimedean, non-trivial absolute value $|~ |$.
A variety over $K$ is an irreducible separated reduced scheme of finite type over $K$ and a curve is a $1$-dimensional variety over $K$.}

\ack{The author would like to thank Matt Baker for the opportunity to visit him for a two months research stay and for suggesting the generalization of the Energy Minimization Principle to the author.
The author is also grateful to Walter Gubler for very carefully reading drafts of this work and for the helpful discussions.}

\section{Non-archimedean curves and their skeleta}

Let $X$ be an algebraic smooth projective curve  $X$ over $K$. 
Then by $\Xan$  we always denote the Berkovich analytification of $X$. 
We briefly recall the construction of this analytification. 

\begin{defn}\label{Xan} 
	For an open affine subset $U=\Spec(A)$ of  $X$, the analytification $U^{\an}$ is the set of all multiplicative seminorms on~$A$ extending the given absolute value $|~|$ on~$K$. We endow the set $U^{\an}$ with the coarsest topology such that $U^{\an}\to \R,~p\mapsto p(a)$ is continuous for every element $a\in A$. 
	By gluing, we get a topological space $\Xan$,
	which is connected,  compact and Hausdorff.
	We call the space $\Xan$ the \emph{(Berkovich) analytification} of $X$ which is a $K$-analytic space in the sense of \cite[\S 3.1]{BerkovichSpectral}.
	
\end{defn}
\begin{bem} 
	Note that the space $\Xan$ is in fact path-connected. 
	Here  a \emph{path} from  $x$ to $y$ is a continuous injective map $\gamma\colon[a,b]\to \Xan$
	with $\gamma(a)=x$ and $\gamma(b)=y$.
	If there is a unique path between two points $x,y\in \Xan$, we write $[x,y]$ for this path.
	We often use the notations $(x,y):=[x,y]\backslash \{x,y\}$, $(x,y]:=[x,y]\backslash \{x\}$ and $[x,y):=[x,y]\backslash \{y\}$.
\end{bem} 

\begin{bem}
	The points of $\Xan$ can be classified in four different types following  \cite[\S 1.4]{BerkovichSpectral}, \cite[\S 2.1]{Th} and \cite[\S 3.5]{BPR1}.
	The points of type I can be identified with the rational points $X(K)$.
	By $I(\Xan)$ we denote the subset of points of type II or III, and by $\HM(\Xan)$ the subset of points of type II, III and IV, i.e.~$\HM(\Xan)=\Xan\backslash X(K)$. For any subset $S$ of $\Xan$, we write $I(S):=S\cap I(\Xan)$ and $\HM(S):=S\cap \HM(\Xan)$.
	The sets $I(\Xan)$ and $X(K)$ are dense in $\Xan$. 
\end{bem}

\begin{bem}\label{boundary} When we talk about the boundary of a subset $W$ of $\Xan$, we always mean (if nothing is stated otherwise) the Berkovich boundary of $W$, which is the topological boundary in $\Xan$.
	For an affinoid domain the Berkovich boundary coincides with Shilov boundary and the limit boundary, and it is always a finite set of  points of type II or III in $\Xan$ (see \cite[Proposition 2.1.12]{Th} for definitions and a proof). If the affinoid domain is strictly affinoid, all boundary points are of type II.
\end{bem}

 Due to the nice properties of the topological space $\Xan$,  finite signed Borel measures are automatically regular.

\begin{prop}\label{Radon}\label{Proho}
	Every finite signed Borel measure on~$\Xan$ is a signed Radon measure.
In particular, every net $(\nu_\alpha)_\alpha$ of probability measures $\nu_\alpha$ on~$\Xan$ has a subnet that converges weakly to a probability measure $\nu$ on~$\Xan$.
\end{prop}
\begin{proof}
The first assertion follows by \cite[Theorem 7.8]{Fo}, as every open subset of the locally compact Hausdorff space $\Xan$ is the countable union of compact sets by \cite[(2.1.5)]{CLD}.
Since every probability measure is so a Radon measure, the second assertion follows by the Prohorov's theorem for nets (see for example \cite[Theorem A.11]{BR}). 
\end{proof}

\begin{nurnum}
Another important property of the analytification $\Xan$ of a smooth projective curve $X$ over $K$ is the existence of so called \emph{skeleta}.
Skeleta are deformation retracts of $\Xan$ and they have the structure of a metric graph.
We refer to  \cite{BPR2} for their definition via  semistable vertex sets.
Without loss of generality all our considered skeleta do not have any loop edges (cf.~\cite[Corollary 3.14]{BPR2}).
Note that their definition of skeleta is consistent  with Thuillier's notion.
For a skeleton $\Gamma$ of $\Xan$ we write $\Gamma_0$ for its vertex set and $\tau_\Gamma$ for its retraction map.
\end{nurnum}

\begin{prop}As sets we have 
$$I(\Xan)=\bigcup_{\Gamma \text{ skeleton of  }\Xan}\Gamma.$$
\end{prop}
\begin{proof}
See \cite[Corollary 5.1]{BPR2}.
\end{proof}

\begin{prop}\label{skeletonprop}
	Let $\Gamma$ be a skeleton of $\Xan$, then the following are true:
	\begin{enumerate}
		\item $\Gamma$ is a connected, compact subset of  points of type II and III and has the structure of a metric graph.
\item For a finite subset $S\subset I(\Xan)$, there is a skeleton $\Gamma'$ of $\Xan$ such that $\Gamma'$ contains $\Gamma$ as a finite metric subgraph and $S\subset \Gamma'$.
		\item For a finite subset $S$ of type II points in $\Gamma$, there is a skeleton of $\Xan$ such that $\Gamma'$ contains $\Gamma$ as a finite metric subgraph with $\Gamma_0\cup S = \Gamma'_0$, i.e.~$\Gamma'$ and $\Gamma$ are equal as sets.
	\end{enumerate}
\end{prop}
\begin{proof}
See \cite[Lemma 3.4]{BPR2}, \cite[Lemma 3.13]{BPR2} and use the last proposition.
\end{proof}

\begin{defn}\label{paf}\label{skeletalmetric}
	With the help of the shortest-path metric on every skeleton and the fact that $I(\Xan)$ can be exhausted by skeleta,
	one can define a metric $\rho$ on~$\HM(\Xan)$ (cf.~\cite[\S 5]{BPR2}), which is called the \emph{skeletal metric}. 
\end{defn}

\begin{defn}
	Let $\Gamma$ be a skeleton of $\Xan$. Then a subset $\Omega$ of $\Gamma$ is a \emph{star-shaped open subset}  of $\Gamma$  if $\Omega$ is a simply-connected open subset of $\Gamma$ and there is a point $x_0\in \Omega$ such that  $\Omega\backslash\{x_0\}$ is a disjoint union of open intervals.
	We call $x_0$ the \emph{center} of $\Omega$.
\end{defn}

\begin{Thm}\label{Str}
	Let $x_0\in \Xan$. There is a fundamental system of open neighborhoods $\{V_\alpha\}$ of $x_0$ of the following form:
	\begin{enumerate}
		\item If $x_0$ is of type I or type IV, then the $V_\alpha$ are open balls.
		\item If $x_0$ is of type III, then the $V_\alpha$ are open annuli with $x_0$ contained in the skeleton of the annulus $V_\alpha$ (cf.~\cite[\S 2]{BPR2}).
		\item If $x_0$ is of type II, then $V_\alpha=\tau_\Gamma^{-1}(\Omega_\alpha)$ for a skeleton $\Gamma$ of $\Xan$ and a star-shaped open subset $\Omega_\alpha$  of $\Gamma$. 
		Hence each $V_\alpha\backslash\{x_0\}$ is a disjoint union of open balls and open annuli.
	\end{enumerate}
\end{Thm}
\begin{proof}
	See \cite[Corollary 4.27]{BPR2}.
\end{proof}
\begin{defn}\label{simple neighborhood}
	An open subset of the described form in Theorem \ref{Str} is  called \emph{simple open}.
\end{defn}
\begin{bem}
	Theorem \ref{Str} implies  directly that $\Xan$ is locally path-connected.
\end{bem}

\section{Subharmonic functions on non-archimedean curves}\label{Thuillier}
Thuillier developed in his thesis \cite{Th} a potential theory on non-archimedean curves, which is based on skeleta.
In this section, we introduce his subharmonic functions on $\Xan$ via his class of smooth functions with their corresponding Laplacian.

\begin{defn}Let  $\Gamma$ be a skeleton of  $\Xan$.
	\begin{enumerate}
		\item A \emph{piecewise affine function} on~$\Gamma$ is a continuous function $F\colon \Gamma\to \R$ such that  $F|_e\circ \alpha_e$ is piecewise affine for every edge $e$ of $\Gamma$, where $\alpha_e$ is an identification of $e$ with a real closed interval.
		\item We define the \emph{outgoing slope} of a piecewise affine function $F$ on~$\Gamma$ at a point $x\in \Gamma$ along a tangent direction $v_e$ at $x$ corresponding to an adjacent edge $e$ as $$d_{v_e}F(x):=\lim_{\varepsilon\to 0}(F|_e\circ \alpha_e)^\prime(\alpha_e^{-1}(x)+\varepsilon).$$

		One obtains a finite measure  on~$\Xan$ by putting $$dd^c F:=\sum_{x\in \Gamma}^{}\left(\sum_{v_e}d_{v_e} F(x)\right)\delta_x,$$
		where $e$ is running over all edges in $\Gamma$ at $x$.
		Since $F$ is piecewise affine, we have  $\sum_{v_e}d_{v_e} F(x)\neq 0$  for only finitely many points in $\Gamma$. 
	\end{enumerate} 		
\end{defn}

\begin{defn}\label{Dlisse} Let $W\subset \Xan$ be open.
A continuous function $f\colon W\to \R$ is called \emph{smooth} if for every point $x\in W$ there is a neighborhood $V$ of $x$ in $W$, a skeleton $\Gamma$ of $\Xan$ and a piecewise affine function $F$ on $\Gamma$ such that 
$$f=F\circ \tau_\Gamma$$ on $V$.
 We denote by $A^0(W)$ the vector space of smooth functions on~$W$, and by $A_c^0(W)$ the subspace of smooth functions on~$W$ with compact support in $W$. 
\end{defn}
\begin{bem}
One should note that these smooth functions are not necessarily smooth in the sense of Chambert-Loir and Ducros from \cite{CLD}. 
In \cite{Wa18} and \cite{Wa} we work with both notions and so smooth functions in the sense of Thuillier from Definition~\ref{Dlisse} are called lisse to distinguish them form those defined by Chambert-Loir and Ducros. 
\end{bem}
\begin{defn}\label{Bem dd^c} 
We write $A^1(W)$ for the set of real measures on~$W$ with discrete support in $I(W)$, and use $A_c^1(W)$ for those with compact support in $W$.   
Then for every smooth function $f\in A^0(W)$, there is a unique real measure $dd^cf$ in $A^1(W)$ such that $$dd^c f= dd^c F$$  whenever  $f=F\circ \tau_{\Gamma}$ for a skeleton $\Gamma$ of $\Xan$ (cf.~\cite[Th\'eor\`eme 3.2.10]{Th}). We call this linear operator $dd^c\colon A^0(W)\to A^1(W)$ the \emph{Laplacian}. Note that $A_c^0(W)$ is mapped to $A^1_c(W)$ under $dd^c$ \cite[Corollaire~3.2.11]{Th}. 
\end{defn}

\begin{prop}\label{ThF}
	For any two points $x,y\in I(\Xan)$ there is a unique smooth function $g_{x,y}\in A^0(\Xan)$ such that 
	\begin{enumerate}
		\item $dd^c g_{x,y} =\delta_x -\delta_y$, and \item $g_{x,y}(x)=0$.
	\end{enumerate}
\end{prop}
\begin{proof}
	See \cite[Proposition 3.3.7]{Th}.
\end{proof}

\begin{defn}
	Let $W$ be an open subset of $\Xan$. 
	We denote by $D^0(W)$ (resp.~$D^1(W)$) the dual of $A_c^1(W)$ (resp.~$A^0_c(W)$). 
\end{defn}

\begin{prop}\label{Hom}
	The map 
	$$D^0(W)\to \Hom(I(W),\R),~ T\mapsto (x \mapsto \langle  T, \delta_x \rangle)$$ is an isomorphism of vector spaces.
\end{prop}
\begin{proof}See \cite[Proposition 3.3.3]{Th}.
\end{proof}
In the following, we always use this identification.
\begin{bem}
	The Laplacian $dd^c\colon A^0_c(W)\to A^1_c(W)$ on an open subset $W\subset \Xan$ leads  naturally by duality to an $\R$-linear operator $$dd^c\colon D^0(W)\to D^1(W),~T\mapsto (g\mapsto \langle dd^c T,g\rangle:=\langle T,dd^c g\rangle)$$ such that the following diagram commutes
\begin{align*}
\begin{xy}
  \xymatrix{
      A^0(W) \ar[r]^{dd^c} \ar[d]   &   A^1(W) \ar[d] \\
      D^0(W) \ar[r]^{dd^c}            &   D^1(W).
  }
\end{xy}
\end{align*}
\end{bem}

\begin{defn}
We say that  a current $T\in D^1(W)$ on an open subset $W$ of $\Xan$ is \emph{positive} if  $ \langle T, g\rangle\geq 0$ for every non-negative smooth function $g\in A^0_c(W)$.
\end{defn}
Before introducing subharmonic functions we recall upper respectively lower semi-continuity.
\begin{bem}\label{usc}
	A function $f\colon W \to [-\infty,\infty)$ on an open subset $W$ of a topological space is \emph{upper semi-continuous}  (shortly \emph{usc}) in a point $x_0$ of $W$ if
	$$\limsup_{x\to x_0} f(x)\leq f(x_0),$$ where the limit superior in this context is defined as
	\begin{align*}\limsup _{x\to x_0}f(x):=\sup_{U\in \US(x_0)}\inf_{x\in U\backslash\{x_0\}}f(x) ,\end{align*}
	where $\US(x_0)$ is any basis of open neighborhoods of $x_0$.
	We say that $f$ is \emph{upper semi-continuous} (shortly \emph{usc}) \emph{on~$W$} if it is usc in all points of $W$.

	Analogously, a function $f\colon W \to (-\infty,\infty]$ on an open subset $W$ of a topological space is \emph{lower semi-continuous} (shortly \emph{lsc}) in a point $x_0$ of $W$ if
	$$\liminf_{x\to x_0} f(x)\geq f(x_0),$$ where the limit inferior in this context is defined as
	\begin{align*}\liminf _{x\to x_0}f(x):=\inf_{U\in \US(x_0)}\sup_{x\in U\backslash\{x_0\}}f(x) ,\end{align*}
	where $\US(x_0)$ is any basis of open neighborhoods of $x_0$.
	We say that $f$ is \emph{lower semi-continuous} (shortly \emph{lsc}) \emph{on~$W$} if it is lsc in all points of $W$.
\end{bem}
\begin{defn} \label{Prop not lisse} Let $W$ be an open subset  of $\Xan$.
	An upper semi-continuous function $f\colon W\to [-\infty,\infty)$ is called \emph{subharmonic} if and only if $f\in D^0(W)$ and $dd^cf\geq 0$.

A continuous function $h\colon W\to \R$ is called \emph{harmonic} if $h$ and $-h$ are subharmonic, i.e.~$dd^c h=0$.
\end{defn}

\begin{bem}\label{BemSH}
	Note that this is not Thuillier's original definition of subharmonic functions, but it is equivalent  by \cite[Th\'eor\`eme 3.4.12]{Th}. 
Baker and Rumely independently introduced subharmonic functions on $\mathbb{P}^{1,\an}$. 
However, their class equals Thuillier's class of subharmonic functions in this special case.

If $f$ is smooth, then $f$ is subharmonic if and only if $dd^cf$ is a positive measure \cite[Proposition 3.4.4]{Th}. 
Moreover, note that a harmonic function $h$ is automatically smooth i.e.~$h\in A^0(W)$ by \cite[Corollaire 3.2.11]{Th}.
\end{bem}

\begin{prop}\label{Max}
	Let $W$ be an open subset of $\Xan$. Then a subharmonic function $f\colon W\to [-\infty,\infty)$ admits a local maximum in a point $x_0$ of $W$ if and only if it is locally constant at $x_0$. 
\end{prop}
\begin{proof}
	See \cite[Proposition 3.1.11]{Th}.
\end{proof}

\begin{prop}\label{sheaf}
	The subharmonic functions form a sheaf on~$\Xan$.
\end{prop}
\begin{proof}
	See  \cite[Corollaire 3.1.13]{Th}.
\end{proof}

\begin{bem}\label{convex}\label{BemSlopesSum}
	Let $f\colon W\to [-\infty,\infty)$ be a subharmonic function on an open subset $W$ of $\Xan$ and
 let $[x_0,y_0]$ be an interval (i.e.~a segment of an edge) in  a skeleton $\Gamma$ of $\Xan$ such that $\tau_\Gamma^{-1}((x_0,y_0))\subset W$.
Then one can show that $f$ is convex restricted to  the relative interior of $I=[x_0,y_0]$ (see for example \cite[Remark 3.1.32]{Wa}).
\end{bem}

\section{Potential kernel}\label{Potentialkernel}
On the way to define Arakelov--Green's functions and prove an Energy Minimization Principle, we have to introduce a lot of other things first. 
Our most fundamental tool is the potential kernel that is a function $g_\zeta(\cdot,y)$ for fixed $\zeta$ and $y$ that inverts the Laplacian in the sense that   $dd^c g_{\zeta}(\cdot,y)=\delta_\zeta-\delta_y$.
A function with this property was already seen in Proposition~\ref{ThF} for $\zeta,y\in I(\Xan)$. 

\begin{defn}\label{pkm}
Let $\Gamma$ be a metric graph. 
For fixed points $\zeta,y\in \Gamma$, let $ g_{\zeta}(\cdot,y)_{\Gamma}\colon \Gamma \to \R_{\geq 0}$ be the unique piecewise affine function on~$\Gamma$ such that 
\begin{enumerate}
	\item $ dd^c g_{\zeta}(\cdot,y)_{\Gamma}=\delta_\zeta-\delta_y$, and
	\item $g_{\zeta}(\zeta,y)_{\Gamma}=0$.
\end{enumerate}
We call $ g_{\zeta}(x,y)_{\Gamma}$ the \emph{potential kernel} on~$\Gamma$.
\end{defn}

\begin{lem}\label{ggraph}
Let $\Gamma$ be a metric graph, then the potential kernel $ g_{\zeta}(x,y)_{\Gamma}$ on~$\Gamma$ is non-negative, bounded, symmetric in $x$ and $y$, and jointly continuous in $x,y,\zeta$. 
For every $\zeta'\in \Gamma$, we have 
$$g_{\zeta}(x,y)_{\Gamma}=g_{\zeta'}(x,y)_{\Gamma}-g_{\zeta'}(x,\zeta)_{\Gamma}-g_{\zeta'}(y,\zeta)_{\Gamma}+g_{\zeta'}(\zeta,\zeta)_{\Gamma}.$$
\end{lem}
\begin{proof}
	Follows by \cite[Proposition 3.3]{BR}.
\end{proof}
Since every skeleton of $\Xan$ has the structure of a metric graph, we can define a potential kernel on every skeleton. 
Using the skeletal metric $\rho\colon \HM(\Xan)\times \HM(\Xan) \to \R_{\geq 0}$ from Definition~\ref{skeletalmetric}, we can extend the potential kernel to all of $\Xan$.

\begin{bem}\label{Wegpunkt}	
Let $V$ be a uniquely path-connected subset of $\Xan$ and let $\zeta$ 	be a point in $V$. 
For two points $x,y\in V$, we denote by $w_\zeta(x,y)$ the unique point in $V$ where the paths $[x,\zeta]$ and $[y,\zeta]$ first meet.
For example, for  a skeleton $\Gamma$ of $\Xan$ and a point $x_0\in \Gamma$,  the subset $\tau_\Gamma^{-1}(x_0)$  is uniquely path-connected.	
We therefore can define for two points $x,y\in \tau_\Gamma^{-1}(x_0)$  the point $w_\Gamma(x,y):= w_{x_0}(x,y)$.
\end{bem}

\begin{defn} Let $\zeta\in I(\Xan)$. We define \emph{potential kernel} $g_{\zeta}\colon \Xan \times \Xan \to (-\infty,\infty]$ corresponding to $\zeta$ by 
\begin{align*}
g_{\zeta}(x,y):= \begin{cases} 
\infty & \mbox{if } (x,y)\in \text{Diag}(X(K)),\\
g_{\zeta}(\tau_{\Gamma}(x),\tau_{\Gamma}(y))_{\Gamma} & \mbox{if } \tau_{\Gamma}(x)\neq \tau_{\Gamma}(y),\\
g_{\zeta}(\tau_{\Gamma}(y), \tau_{\Gamma}(x))_{\Gamma}+\rho((w_{\Gamma}(x,y), \tau_{\Gamma}(y))& \mbox{else}\\
\end{cases}
\end{align*} 
for a skeleton $\Gamma$  of $\Xan$ containing $\zeta$ and the skeletal metric $\rho\colon \HM(\Xan)\times \HM(\Xan) \to \R_{\geq 0}$.
\end{defn}

\begin{prop}
The function $g_{\zeta}$ is well-defined for every $\zeta\in I(\Xan)$. 
\end{prop}
\begin{proof}
We have to show that $g_{\zeta}$ is independent of the skeleton $\Gamma$. 
Thus we consider $(x,y)\notin \text{Diag}(X(K))$.
 Let $\Gamma_1$ and $\Gamma_2$ be two skeleta containing $\zeta$, and we may assume that $\Gamma_1\subset \Gamma_2$.
Since $\Gamma_1$ is already a skeleton of $\Xan$, $\Gamma_2$ arises by just adding additional edges and vertices to  $\Gamma_1$ without getting new loops or cycles. 
Working inductively, we may assume that $\Gamma_2$ equals to the graph $\Gamma_1$ and one new edge $e$  attached to a vertex $z$ in $\Gamma_1$. 

Note that for every $w\in \Gamma_1$, due to uniqueness of the potential kernel and because its Laplacian is supported on~$\{\zeta,w\}$, we have
\begin{align}\label{Note}
 g_\zeta(\cdot,w)_{\Gamma_1}\equiv g_\zeta(\cdot,w)_{\Gamma_2} \text{ on~$\Gamma_1$, and }   g_\zeta(\cdot,w)_{\Gamma_2}\equiv g_\zeta(z,w)_{\Gamma_2}	\text{ on~$e$.}
\end{align}
Hence $g_\zeta(v,w)_{\Gamma_1}= g_\zeta(v,w)_{\Gamma_2} $ for every pair $(v,w)\in \Gamma_1\times\Gamma_1$.

 First, we consider $(x,y)$ with $\tau_{\Gamma_1}(x)\neq \tau_{\Gamma_1}(y)$. Then  automatically $\tau_{\Gamma_2}(x)\neq \tau_{\Gamma_2}(y)$ and $\tau_{\Gamma_1}(x)=\tau_{\Gamma_2}(x)$ or $\tau_{\Gamma_1}(y)=\tau_{\Gamma_2}(y)$.
Without loss of generality $\tau_{\Gamma_1}(y)=\tau_{\Gamma_2}(y)$.
As argued in (\ref{Note}) and using symmetry, we get 
 \begin{align*}
g_{\zeta}(\tau_{\Gamma_1}(x),\tau_{\Gamma_1}(y))_{\Gamma_1} & = g_{\zeta}(\tau_{\Gamma_1}(x),\tau_{\Gamma_1}(y))_{\Gamma_2}\\
&=  g_{\zeta}(\tau_{\Gamma_2}(x),\tau_{\Gamma_1}(y))_{\Gamma_2}\\
& =g_{\zeta}(\tau_{\Gamma_2}(x),\tau_{\Gamma_2}(y))_{\Gamma_2}.
 \end{align*}
Note for the second equation that either $\tau_{\Gamma_1}(x)=\tau_{\Gamma_2}(x)$ or $\tau_{\Gamma_1}(x)=z$ and $\tau_{\Gamma_2}(x)\in e$.

Now consider the case  $\tau_{\Gamma_1}(x)= \tau_{\Gamma_1}(y)$, and we set  $w:=w_{\Gamma_1}(x,y)$ (cf.~Remark~\ref{Wegpunkt}). 
Then (\ref{Note})  implies
\begin{align}\label{zwei} g_\zeta(\tau_{\Gamma_1}(x),\tau_{\Gamma_1}(y))_{\Gamma_1}=g_\zeta(\tau_{\Gamma_1}(x),\tau_{\Gamma_1}(y))_{\Gamma_2} =g_\zeta(\tau_{\Gamma_2}(x),\tau_{\Gamma_1}(y))_{\Gamma_2}.\end{align}
Note again that  that either $\tau_{\Gamma_1}(x)=\tau_{\Gamma_2}(x)$ or $\tau_{\Gamma_1}(x)=z$ and $\tau_{\Gamma_2}(x)\in e$.
In the case $\tau_{\Gamma_2}(x)=\tau_{\Gamma_2}(y)= \tau_{\Gamma_1}(y)=\tau_{\Gamma_1}(x)$, then the line above implies the claim.

If  $\tau_{\Gamma_2}(x)=\tau_{\Gamma_2}(y)\neq  \tau_{\Gamma_1}(y)=\tau_{\Gamma_1}(x)$, then  $\tau_{\Gamma_1}(y)=\tau_{\Gamma_1}(x)=z$ and
 we have 
$$\rho(w,\tau_{\Gamma_1}(x))= \rho(w,\tau_{\Gamma_2}(x)) + \rho(\tau_{\Gamma_2}(x),\tau_{\Gamma_1}(x)).$$  Identity  (\ref{zwei}) implies
\begin{align*}
g_\zeta(\tau_{\Gamma_1}(x),\tau_{\Gamma_1}(x))_{\Gamma_1} &=g_\zeta(\tau_{\Gamma_2}(x),\tau_{\Gamma_1}(x))_{\Gamma_2}\\
&=g_\zeta(\tau_{\Gamma_1}(x),\tau_{\Gamma_2}(x))_{\Gamma_2}\\
&=g_\zeta(\tau_{\Gamma_2}(x),\tau_{\Gamma_2}(x))_{\Gamma_2} - \rho(\tau_{\Gamma_1}(x),\tau_{\Gamma_2}(x)),\end{align*}
where we use 
for the last equation that $g_\zeta(\cdot,\tau_{\Gamma_2}(x))_{\Gamma_2}$ restricted to the path $[z,\tau_{\Gamma_2}(x)]$ is affine with slope $1$.
Adding these two equations up, we get 
$$g_\zeta(\tau_{\Gamma_1}(x),\tau_{\Gamma_1}(x))_{\Gamma_1}+ \rho(w,\tau_{\Gamma_1}(x))  =  g_\zeta(\tau_{\Gamma_2}(x),\tau_{\Gamma_2}(x))_{\Gamma_2} + \rho(w,\tau_{\Gamma_2}(x))$$ as we desired.

If $\tau_{\Gamma_2}(x)\neq \tau_{\Gamma_2}(y)$, then $z=\tau_{\Gamma_1}(x)=\tau_{\Gamma_1}(y)$ and $w=\tau_{\Gamma_2}(x)$ or  $w= \tau_{\Gamma_2}(y)$. Without loss of generality, $w=\tau_{\Gamma_2}(y)$. 
The potential kernel
 $g_\zeta(\cdot,\tau_{\Gamma_2}(x))_{\Gamma_2}$ restricted to the path $[z,\tau_{\Gamma_2}(x)]$ (which contains $w=\tau_{\Gamma_2}(y)$) is affine with slope $1$. Hence (\ref{zwei}) and symmetry yield $$g_\zeta(\tau_{\Gamma_1}(x),\tau_{\Gamma_1}(y))_{\Gamma_1} =g_\zeta(\tau_{\Gamma_1}(y),\tau_{\Gamma_2}(x))_{\Gamma_2}=g_\zeta(\tau_{\Gamma_2}(y), \tau_{\Gamma_2}(x))_{\Gamma_2} - \rho(\tau_{\Gamma_1}(y),w).$$
Consequently, $g_\zeta(x,y)$ is well-defined.
\end{proof}

\begin{bem} \label{VGL}
	If $X=\mathbb{P}^{1}$, it is easy to see that the function $g_{\zeta}$ coincides with the potential kernel $j_{\zeta}$ from \cite[\S 4.2]{BR} for every $\zeta\in I(\Xan)$. 
\end{bem}

\begin{lem}\label{g} Fix $\zeta\in I(\Xan)$. As a  function of two variables  $g_{\zeta}(x,y)$ satisfies the following properties:
 \begin{enumerate}

\item It is non-negative and $g_{\zeta}(\zeta,y)=0$.
\item $g_{\zeta}(x,y)=g_{\zeta}(y,x)$.
\item For every $\zeta'\in I(\Xan)$, we have
\begin{align}\label{independent}
g_{\zeta}(x,y)=g_{\zeta'}(x,y)-g_{\zeta'}(x,\zeta)-g_{\zeta'}(y,\zeta)+g_{\zeta'}(\zeta,\zeta).
\end{align}
\item It is finitely valued and continuous  off the diagonal and it is lsc on~$\Xan\times \Xan$ (where we understand $\Xan\times \Xan$  set theoretically and endowed with the product topology).
\end{enumerate}
\end{lem}
\begin{proof}
All properties follow by construction and the properties of the potential kernel on a metric graph from Lemma~\ref{ggraph}.
Note for the third assertion that we choose a skeleton such that $\zeta,\zeta'\in \Gamma$. 
A detailed proof od property iv) can be found in \cite{Wa}.
\end{proof}

\begin{prop}\label{G1} For  fixed points $\zeta\in I(\Xan)$ and $y\in \Xan$, we consider the function $G_{\zeta,y}:=g_{\zeta}(\cdot,y)\colon \Xan \to (-\infty,\infty]$.
Then $G_{\zeta,y}$ defines a current in $D^0(\Xan)$ with $$dd^c G_{\zeta,y}=  \delta _\zeta - \delta_y.$$
Moreover, the following hold: 
\begin{enumerate}
	\item If $y$ is of type II or III, then $G_{\zeta,y}\in A^0(\Xan)$ and  coincides with $g_{y,\zeta}$ from Proposition~\ref{ThF}.
	\item If $y$ is of type IV, the function  $G_{\zeta,y}$ is finitely valued and continuous on~$\Xan$. 
	\item If $y$ is of type I, then $G_{\zeta,y}$ is finitely valued on 
	$\Xan\backslash \{y\}$ and continuous on~$\Xan$ when we endow $(-\infty,\infty]$ with the topology of a half-open interval.
\end{enumerate}
Hence  $G_{\zeta,y}$ is  subharmonic on~$\Xan\backslash\{y\}$ for every fixed $y\in \Xan$.

\end{prop}

\begin{proof}First, note that by construction $G_{\zeta,y}(x)=\infty$ if and only if $x=y\in X(K)$.
Thus the restriction of $G_{\zeta,y}$ to $I(\Xan)$ is always finite, and so $G_{\zeta,y}$ defines a current in $D^0(\Xan)$.
Here, one should have in mind that the vector space $D^0(\Xan)$ is isomorphic to the vector space $\Hom(I(\Xan),\R)$ endowed with the pointwise convergence (see Proposition~\ref{Hom}). We always use this identification.

Let $\Gamma$ always be  a skeleton that contains $\zeta$.
To calculate the Laplacian, we first consider a point $y\in I(\Xan)$.
We may  extend $\Gamma$ such that $y\in \Gamma$.
Then  $G_{\zeta,y}=g_{\zeta}(\cdot, y)_\Gamma\circ \tau_{\Gamma}$ on~$\Xan$ since
$$\rho(w_\Gamma(x,y),y)=\rho(\tau_\Gamma(x),\tau_\Gamma(y))=0$$ if $\tau_{\Gamma}(x)=\tau_{\Gamma}(y)=y$. 
Since the potential kernel $g_{\zeta}(\cdot,y)_{\Gamma}$  is the unique piecewise affine function on the metric graph $\Gamma$ such that $dd^c g_{\zeta}(\cdot,y)_{\Gamma} =  \delta_\zeta-\delta_y$ and $g_{\zeta}(\zeta,y)_{\Gamma}=0$, we have $G_{\zeta,y}=g_{\zeta}(\cdot,y)\in A^0(\Xan)$ and $dd^c G_{\zeta,y}= \delta_\zeta -  \delta_y $ on~$\Xan$ by the construction of the Laplacian.
In particular, the function $G_{\zeta,y}$ is continuous on~$\Xan$. 
Uniqueness in Proposition~\ref{ThF} implies that $G_{\zeta,y}$  coincides with $g_{y,\zeta}$.

Now consider an arbitrary $y\in \Xan\backslash I(\Xan)$ and let $(y_n)_{n\in \N}$ be a sequence of points $y_n\in I(\Xan)$ converging to $y$. 
Then $g_\zeta(\cdot,y_n)$ converges to $g_\zeta(\cdot,y)$ in the topological  vector space $\Hom(I(\Xan),\R)\simeq D^0(\Xan)$, i.e.~for every fixed point $x\in I(\Xan)$ we have $g_\zeta(x,y_n)=g_\zeta(y_n,x)$ converges to $g_\zeta(x,y)=g_\zeta(y,x)$ for $n\to \infty$ since $g_\zeta(\cdot,x)$ is smooth, and so continuous.
The differential operator $dd^c\colon D^0(\Xan)\to D^1(\Xan)$ is continuous by \cite[Proposition 3.3.4]{Th}, and hence $dd^c g_\zeta(\cdot,y) = \delta_{\zeta}- \delta_y.$

If $y$ is a point of type I or IV, the connected component $U$ of $\Xan\backslash \Gamma$ containing $y$ is an open ball. 
For a type IV point $y$, we have  \begin{align*}G_{\zeta,y}(x)&=g_\zeta(\tau_\Gamma(y),\tau_\Gamma(y))_\Gamma+\rho(w_{\Gamma}(x,y),\tau_\Gamma(x))\\ &=g_\zeta(\tau_\Gamma(y),\tau_\Gamma(y))_\Gamma+\rho(w_{\tau_\Gamma(y)}(x,y),\tau_\Gamma(x))
\end{align*} for every $x\in U$. Note that $\tau_\Gamma(x)=\tau_\Gamma(y)$ and $\zeta\in \Gamma$.
If $y$ is of type I, we have this identity on~$U\backslash\{y\}$.
Since the path distance metric $\rho$ is continuous on~$U$, it follows that $G_{\zeta,y}$  is continuous on~$U$ in both cases with $\lim_{x\to y}G_{\zeta,y}(x)=G_{\zeta,y}(y)=\infty$ if $y$ is of type~I. 

In particular, $G_{\zeta,y}$ is upper semi-continuous on~$\Xan\backslash \{y\}$ with $dd^cG_{\zeta,y}=\delta_{\zeta}- \delta_y$ for every fixed $y\in \Xan$.
Hence $G_{\zeta,y}$ is subharmonic on~$\Xan\backslash \{y\}$. 
\end{proof}

To introduce a capacity theory  and define potential functions  on~$\Xan$ in the following sections, we define  a potential 
 kernel $g_\zeta(x,y)$ for every point $\zeta\in \Xan$ (cf.~\cite[\S 4.4]{BR}).
In \cite{BR}, this is done with the help of the Gauss point. In our case, we have to fix a base point for the definition.

\begin{defn}\label{Kernel}
Fix  $\zeta_0\in I(\Xan)$. 
We define  $g_{\zeta_0}\colon \Xan\times \Xan\times \Xan\to  [-\infty,\infty]$  as $g_{\zeta_0}(\zeta,x,y)=\infty $ if $x=y=\zeta\in X(K)$ and else as 
$$g_{\zeta_0}(\zeta,x,y):=g_{\zeta_0}(x,y)-g_{\zeta_0}(x,\zeta)- g_{\zeta_0}(y,\zeta).$$
\end{defn}
\begin{kor}\label{3Var}
For  fixed points $\zeta_0\in I(\Xan)$ and $\zeta,y\in \Xan$, the potential kernel $g_{\zeta_0}(\zeta,\cdot,y)$  defines a current in $D^0(\Xan)$ with
$$dd ^c g_{\zeta_0}(\zeta,\cdot,y)= \delta_\zeta-\delta_y,$$
and  it extends $g_{\zeta}(x,y)$  in the following way 
$$ g_{\zeta}(\zeta,x,y)=g_{\zeta}(x,y)$$ if $\zeta\in I(\Xan)$.
\end{kor}
\begin{proof} Follows directly by construction,
Proposition~\ref{G1} and linearity of $dd^c$.
\end{proof}

\section{Capacity theory}\label{Capacity}

The main goal of this paper is to prove an analogue of the Energy Minimization Principle. 
On this way, we need to prove some partial results as for example Frostman's theorem. 
One of the tools of showing Frostman's theorem is capacity. 
We therefore introduce capacity analogously as in \cite[\S 6.1]{BR}, show all needed properties and compare our notion with Thuillier's capacity in \cite[\S 3.6.1]{Th}. 

\begin{defn}\label{DefCap} Let $\zeta_0\in I(\Xan)$ be a fixed base point. 
 Then for a  point $\zeta\in \Xan$ and for a probability measure $\nu$ on $\Xan$ with $\supp(\nu)\subset \Xan\backslash \{\zeta\}$, we define the \emph{energy integral} as 	$$I_{\zeta_0,\zeta}(\nu):=\int\int g_{\zeta_0}(\zeta,x,y)~d\nu(x)d\nu(y).$$
Recall from Definition~\ref{Kernel} the extended potential kernel $g_{\zeta_0}(\zeta,\cdot,\cdot)$, which is lower semi-continuous on~$\Xan\backslash \{\zeta\}\times\Xan\backslash \{\zeta\}$ by Lemma~\ref{g} and Proposition~\ref{G1}. Hence the Lebesgue integral with respect to $\nu$ is well-defined.

With the help of the energy integral, one can introduce \emph{capacity} of a proper $E$ of $\Xan$ with respect to $\zeta\in \Xan\backslash E$ as 
$$\gamma_{\zeta_0,\zeta}(E):=e^{- \inf_\nu I_{\zeta_0,\zeta}(\nu)}$$ where $\nu$ varies over all probability measures supported on~$E$.
We say that $E$ \emph{has positive capacity} if there is a $\zeta_0\in I(\Xan)$ and a point $\zeta\in \Xan\backslash E$ such that $\gamma_{\zeta_0,\zeta}(E)>0$, i.e.~there exists a probability measure $\nu$ supported on~$E$ with $I_{\zeta_0,\zeta}(\nu)< \infty$.
	Otherwise, we say that $E$  \emph{has capacity zero}.
\end{defn}
\begin{bem}\label{Ecompact}
It follows from the definition of the capacity of $E$ with respect to $\zeta \in \Xan \backslash E$ that 
$$\gamma_{\zeta_0,\zeta}(E)=\sup_{E'\subset E,~ E' \text{ compact}} \gamma_{\zeta_0,\zeta}(E').$$
\end{bem}
\begin{lem} \label{CapB} Positive capacity is independent of the choice of the chosen base point $\zeta_0$.
\end{lem}
\begin{proof}
Consider $\zeta_0,\zeta_0'\in I(\Xan)$, a proper subset $E$ of $\Xan$, a point $\zeta\in \Xan\backslash E$ and a probability measure $\nu$ supported on~$E$. 
We show that $I_{\zeta_0,\zeta}(\nu)$  is finite if and only if $I_{\zeta'_0,\zeta}(\nu)$ is finite.
Using Definition~\ref{Kernel} and Lemma~\ref{g}, we obtain for every $x,y\in E$ (note that $\zeta\notin E$)
$g_{\zeta_0}(\zeta,x,y)=  g_{\zeta'_0}(\zeta,x,y) + 2g_{\zeta'_0}(\zeta,\zeta_0)- g_{\zeta'_0}(\zeta_0,\zeta_0)$ where the last two terms are finite for all $\zeta\in \Xan\backslash E$.
Considering the energy integrals, we get
\begin{align*}
I_{\zeta_0,\zeta}(\nu)=I_{\zeta'_0,\zeta}(\nu)+ 2g_{\zeta'_0}(\zeta,\zeta_0)- g_{\zeta'_0}(\zeta_0,\zeta_0).
\end{align*}
Hence they differ by a finite constant.
\end{proof}
For the rest of the section, we therefore just fix a base point $\zeta_0\in I(\Xan)$. 
\begin{bem} \label{CapP}
	Let $E$ be a proper subset of $\Xan$ and let $\nu$  be a probability measure supported on~$E$. 
	Then for every $\zeta\in \Xan\backslash E$ 
	 \begin{align*}
I_{\zeta_0,\zeta}(\nu)&=	\int\int g_{\zeta_0}(\zeta,x,y)~d\nu(x)d\nu(y)\\
& = \int\int g_{\zeta_0}(x,y)~d\nu(x)d\nu(y)-2\int g_{\zeta_0}(x,\zeta)~d\nu(x),
	\end{align*}
 where the last term of the right hand side is finite since $g_{\zeta_0}(\cdot,\zeta)$ is continuous on the compact set $\supp(\nu)$ by Proposition~\ref{G1}.
 Thus
	 $I_{\zeta_0,\zeta}(\nu)$ is finite  if and only if $I_{\zeta_0,\xi}(\nu)$ is finite for  every point $\xi\in \Xan\backslash E$.
\end{bem}

\begin{lem} \label{CapII}If $E$ is a proper subset of $\Xan$ containing a point of $\HM(\Xan)$, then $E$ has positive capacity.
\end{lem}
\begin{proof} Choose a point $\zeta\in \Xan \backslash E$ and assume there is a point $z\in \HM(\Xan)\cap E$. Then the Dirac measure $\nu:=\delta_z$ is a probability measure supported on~$E$ and 
	$$I_{\zeta_0,\zeta}(\nu)=\int\int g_{\zeta_0}(\zeta,x,y)~d\nu(x)d\nu(y)=g_{\zeta_0}(z,z)-2g_{\zeta_0}(\zeta,z)<\infty$$ due to $z\in \HM(\Xan)$.
\end{proof}
Note that  $I_{\zeta_0,\zeta_0}(\nu)$ is also well-defined for a probability measure $\nu$ supported on~$\Xan$ with $\zeta_0\in\supp(\nu)$ as 
$$I_{\zeta_0,\zeta_0}(\nu)=\int\int g_{\zeta_0}(\zeta_0,x,y)~d\nu(x)d\nu(y)=\int\int g_{\zeta_0}(x,y)~d\nu(x)d\nu(y)$$ by Corollary~\ref{3Var} and $g_{\zeta_0}$ is lsc on~$\Xan\times\Xan$ by Lemma~\ref{g}.
 
\begin{lem}\label{LC} Let $\zeta$ be a point in $\Xan$, let $E$ be  a subset of $\Xan\backslash \{\zeta\}$ that has capacity zero and let $\nu$ be a  probability measure on~$\Xan$.
If \begin{enumerate}
\item $\supp(\nu)\subset\Xan\backslash \{\zeta\}$ with $I_{\zeta_0,\zeta}(\nu)< \infty$ for some base point $\zeta_0\in I(\Xan)$, or
\item $\zeta\in I(\Xan)$ with $I_{\zeta,\zeta}(\nu)< \infty$,
\end{enumerate}
 then $\nu(E)=0$.
\end{lem}
\begin{proof} The proof is analogous to  \cite[Lemma 6.16]{BR}. 
	Note that $g_{\zeta_0}(\cdot,\zeta)$ is continuous on the compact set $\supp(\nu)$ and $g_{\zeta_0}(x,y)$ as a function of two variables is lsc on~$\supp(\nu)\times \supp(\nu)$ (Proposition~\ref{G1} and Lemma~\ref{g}).
	Hence the extended potential kernel $ g_{\zeta_0}(\zeta,x,y)=g_{\zeta_0}(x,y)-g_{\zeta_0}(x,\zeta)-g_{\zeta_0}(y,\zeta)$ is bounded from below on~$\supp(\nu)\times \supp(\nu)$  by a constant if i) is satisfied.
If $\zeta\in I(\Xan)$, then the function $ g_{\zeta}(\zeta,x,y)= g_{\zeta}(x,y)$ (cf.~Corollary~\ref{3Var}) is lsc on~$\Xan\times\Xan$ by Lemma~\ref{g}, and so also bounded from below on~$\supp(\nu)$.
In both cases let  $C$ be this constant.
	If $\nu(E)>0$, then there is a compact subset $e$ of $E$ such that $\nu(e)>0$. 
	Consider the probability measure $\omega:=(1/\nu(e))\cdot \nu|_e$ on~$e$.
	Then 
	\begin{align*}
	I_{\zeta_0,\zeta}(\omega)&=\int\int g_{\zeta_0}(\zeta,x,y)~d\omega(x)d\omega(y)\\
	&=\int\int(g_{\zeta_0}(\zeta,x,y)-C)~d\omega(x)d\omega(y) +\int\int C ~d\omega(x)d\omega(y)\\
	&\leq \frac{1}{\nu(e)^2}\cdot \int\int (g_{\zeta_0}(\zeta,x,y)-C)~d\nu(x)d\nu(y) + C\\
	&=\frac{1}{\nu(e)^2}\cdot \int\int g_{\zeta_0}(\zeta,x,y)~d\nu(x)d\nu(y) - \frac{1}{\nu(e)^2}\cdot \int\int C~d\nu(x)d\nu(y)+C	\\
	&= \frac{1}{\nu(e)^2}\cdot I_{\zeta_0,\zeta}(\nu) - \frac{\nu(E)^2}{\nu(e)^2}\cdot C+C	< \infty
	\end{align*}
	contradicting that $E$ has capacity zero. Note that in case ii) we have  $\zeta_0=\zeta$ in the calculation.
\end{proof}

\begin{kor}\label{CC} Let $\zeta$ be a point in $\Xan$ and let $E_n$ be a countable collection of Borel sets in $\Xan\backslash \{\zeta\}$ such that $E_n$ has capacity zero for every $n\in \N$. Then the set $E:=\bigcup_{n\in \N}E_n$ has capacity zero.
\end{kor}
\begin{proof} 
	Assume $E$ has positive capacity, i.e.~there is a $\zeta\in \Xan\backslash E$ and a probability measure $\nu$ supported on~$E$ such that $I_{\zeta_0,\zeta}(\nu)< \infty$.
	The set $E$ is measurable since all $E_n$ are, and $\sum_{n\in \N} \nu(E_n)\geq \nu(E)=1.$
	Thus there has to be an $E_n$ such that $\nu(E_n)>0$ contradicting Lemma~\ref{LC}.
\end{proof}

\begin{bem}
	Thuillier introduced in \cite[\S 3.6.1]{Th} relative capacity in an open subset $\Omega$ of $\Xan$ with a non-empty boundary  $\partial \Omega \subset I(\Xan)$.
	The capacity of a compact subset $E$ of $\Omega$ is then defined as 
	$$C(E,\Omega)^{-1}:=\left(\inf_{\nu}\int_E\int_E -g_x(y)~d\nu(x)d\nu(y)\right)\in [0,\infty]$$ where $\nu$ runs over all probability measures supported on~$E$.
	Here $g_x\colon \Omega\to [-\infty,0)$ for $x\in \Omega$ is the unique subharmonic function on~$\Omega$ such that
	\begin{enumerate}
		\item  $dd^c g_x =\delta_x$, and 
		\item   $\lim_{y\in \Omega,~y\to \zeta}g_x(y)=0$
	\end{enumerate} for every $\zeta\in \partial \Omega$ (see \cite[Lemma 3.4.14]{Th}). 
	This notion of relative capacity can be extended canonically to all subsets of $\Omega$ by 
	 $$C(E,\Omega):= \sup_{E'\subset E \text{ compact }} C(E',\Omega).$$
	 
One can show that when $\partial \Omega=\{\zeta\}\subset I(\Xan)$, a subset $E$ of $\Omega$ has positive capacity (as defined in Definition~\ref{DefCap}) if and only if $C(E,\Omega)>0$ (cf.~\cite[Proposition 3.2.19]{Wa}).
\end{bem}

\section{Potential functions}\label{Potentialfunctions}
With the help of the potential kernel from Section~\ref{Potentialkernel}, one can introduce potential functions on~$\Xan$ attached to a finite signed Borel measure.
Baker and Rumely defined these functions on the Berkovich projective line  $\mathbb{P}^{1,\an}$ in \cite[\S 6.3]{BR}.
For the generalization to $\Xan$, we have to fix a type II or III point $\zeta_0$ serving as a base point as the Gauss point does for $\mathbb{P}^{1,\an}$.
We define potential functions with respect to this base point and use them to define Arakelov--Green's functions in Section~\ref{Arakelov}.
Later in Lemma~\ref{L-ind}, we see that the definition of the Arakelov--Green's functions is independent of this choice.

\begin{defn} Let $\zeta_0$ be  a chosen base point in $I(\Xan)$ and  let $\nu$ be any finite signed Borel measure on~$\Xan$.
	For every $\zeta\in I(\Xan)$ or $\zeta\notin \supp(\nu)$, we define the corresponding  \emph{potential function} as   $$u_{\zeta_0,\nu}(x,\zeta):= \int_{\Xan} g_{\zeta_0}(\zeta,x,y)~d\nu(y)$$ for every $x\in \Xan$.
Here $g_{\zeta_0}(\zeta,x,y)$ is the potential kernel defined in Definition~\ref{Kernel}.
\end{defn}

\begin{lem}
Let $\zeta_0$ be  a chosen base point in $I(\Xan)$ and  let $\nu$ be any finite signed Borel measure on~$\Xan$.
	For every $\zeta\in I(\Xan)$ or $\zeta\notin \supp(\nu)$, the function  $u_{\zeta_0,\nu}(\cdot,\zeta)$ is well-defined on~$\Xan$ with values in $\R\cup \{\pm\infty\}$ and we can write 
\begin{align}\label{C}
u_{\zeta_0,\nu}(\cdot,\zeta) 
	& = \int g_{\zeta_0}(\cdot,y)~d\nu(y)-\nu(\Xan)g_{\zeta_0}(\cdot,\zeta)+ C_{\zeta_0,\zeta}
\end{align}
on $\Xan$ for a finite constant $C_{\zeta_0,\zeta}$. 

\end{lem}
\begin{proof}
	By the definition of the potential kernel $g_{\zeta_0}(\zeta,x,y)$, we get for every $x\in \Xan$
	\begin{align*}
	u_{\zeta_0,\nu}(x,\zeta) 
	& = \int g_{\zeta_0}(x,y)~d\nu(y)-\nu(\Xan)g_{\zeta_0}(x,\zeta)- \int g_{\zeta_0}(y,\zeta)~d\nu(y).
	\end{align*}
Since $g_{\zeta_0}(\cdot,\zeta)$ is continuous on the compact subset $\supp(\nu)$ if $\zeta\in I(\Xan)$ or $\zeta\notin \supp(\nu)$ (cf.~Proposition~\ref{G1}), the last term is always a finite constant, and so we get the description in (\ref{C}) with $C_{\zeta_0,\zeta}:=- \int g_{\zeta_0}(y,\zeta)~d\nu(y)$.

To prove that $u_{\zeta_0,\nu}(\cdot,\zeta)$  is well-defined, we have to show that $\infty-\infty$ or $-\infty+\infty$ cannot occur. 
	
If $\zeta\in I(\Xan)$, then $g_{\zeta_0}(x,\zeta)$ is finite for every $x\in \Xan$, and so $u_{\zeta_0,\nu}(\cdot,\zeta)$  is well-defined.
	
Next, we consider $\zeta\notin \supp(\nu)$. For every $x\neq \zeta$, we know that $g_{\zeta_0}(x,\zeta)$ is finite as well, and so 
	$\infty-\infty$ or $-\infty+\infty$ cannot occur.
	It remains to show that the function is well-defined in
	 $x=\zeta \notin \supp(\nu)$. Since  $g_{\zeta_0}(x,\cdot)$ is continuous on the compact subset $\supp(\nu)$ as $x\notin\supp(\nu)$,
    the first term $\int g_{\zeta_0}(x,y)~d\nu(y)$ is finite, and so $u_{\zeta_0,\nu}(x,\zeta) $ is well-defined in $x=\zeta$.
\end{proof}
\begin{bem} \label{Bem u} Let $\zeta_0'$ be another chosen base point in  $I(\Xan)$. Then  Lemma~\ref{g} implies that
for every  $\zeta\in I(\Xan)$ or $\zeta\notin \supp(\nu)$ and $x\in \Xan$ we have
	\begin{align*}
	u_{\zeta_0,\nu}(x,\zeta)= u_{\zeta'_0,\nu}(x,\zeta)+2\nu(\Xan)g_{\zeta'_0}(\zeta,\zeta_0)- \nu(\Xan)g_{\zeta'_0}(\zeta_0,\zeta_0),
	\end{align*} i.e.~the corresponding potential function differ by a constant depending on~$\zeta'_0,\zeta_0$ and $\zeta$.
\end{bem}
\begin{lem}\label{u0}
	Let $\zeta_0$ be  a chosen base point in $I(\Xan)$ and  let $\nu$ be any finite signed Borel measure on~$\Xan$.
	For every skeleton $\Gamma$ of $\Xan$ or every  path $\Gamma=[z,\omega]\subset \HM(\Xan)$, and for every  $\zeta\in I(\Xan)$ or $\zeta\notin \supp(\nu)$, the restriction of $u_{\zeta_0,\nu}(\cdot,\zeta)$ to $\Gamma$ is finite and continuous.
\end{lem} 
\begin{proof}
	First, we consider a skeleton $\Gamma$ of $\Xan$.
	We may assume $\zeta_0\in \Gamma$ by  Remark~\ref{Bem u}.
	Note that the potential kernel satisfies by construction  a retraction formula as in \cite[Proposition 4.5]{BR}, i.e.
\begin{align}\label{Retraction}g_{\zeta_0}(x,y)=g_{\zeta_0}(x,\tau_\Gamma(y))_\Gamma=g_{\zeta_0}(x,\tau_\Gamma(y))\end{align}
 for every $x\in \Gamma$ and $y\in \Xan$. Furthermore, recall the description of $u_{\zeta_0,\nu}(\cdot,\zeta) $ in (\ref{C}).
	
	Then for every  $x\in \Gamma$
	\begin{align*}
	u_{\zeta_0,\nu}(x,\zeta) &=  \int_{\Xan} g_{\zeta_0}(x,y)~d\nu(y)- \nu(\Xan) g_{\zeta_0}(x,\zeta)+ C_{\zeta_0,\zeta}\\
	&= \int_{\Xan} g_{\zeta_0}(x,\tau_{\Gamma}(y))_\Gamma~d\nu(y)-\nu({\Xan})g_{\zeta_0}(x,\tau_{\Gamma}(\zeta))_\Gamma+ C_{\zeta_0,\zeta}\\
	&= \int_{\Gamma} g_{\zeta_0}(x,t)_\Gamma~d((\tau_\Gamma)_*\nu)(t)-\nu({\Xan})g_{\zeta_0}(x,\tau_{\Gamma}(\zeta))_\Gamma+ C_{\zeta_0,\zeta}.
	\end{align*} 
	The first term is finite and continuous by Lemma~\ref{ggraph} and the second one is as well by Lemma~\ref{G1}.
	Hence $u_{\zeta_0,\nu}(\cdot,\zeta) $ is finite and continuous on~$\Gamma$.
	
	In the following, we consider a path $\Sigma:=[z,\omega]$. 
	Recall that $\HM(\Xan)$ is the set of points of type II, III and IV, and every point of type IV has only one tangent direction in $\Xan$ \cite[Lemma 5.12]{BPR2}. 
	We already know that $u_{\zeta_0,\nu}(\cdot,\zeta)$ restricted to every skeleton is finite and continuous.
	Moreover, every path $[z,\omega]$ for $z,\omega\in I(\Xan)$ lies in some skeleton.
	Thus  it remains to consider paths of the form $[z,\tau_\Gamma(z)]$ for a type IV point $z$ and an arbitrary large skeleton $\Gamma$ of $\Xan$. 
	From now on let $\Sigma$ be the considered path $[z,\omega]$ with $\omega:=\tau_\Gamma(z)$.
	
     Let $\zeta_0$ be some base point in $I(\Xan)\cap \Gamma$, which we may choose  that way by  Remark~\ref{Bem u}. 
     Again, we  consider each term of 
	$$u_{\zeta_0,\nu}(x,\zeta) =  \int_{\Xan} g_{\zeta_0}(x,y)~d\nu(y)- \nu({\Xan}) g_{\zeta_0}(x,\zeta)+ C_{\zeta_0,\zeta}$$ 
	 for  $x\in \Sigma$ separately. 
	The second term is finite and continuous in $x$ by Proposition~\ref{G1} (note that $\Sigma\cap X(K)=\emptyset$).

	It remains to consider the first term.
	Let $V$ be the connected component of $\Xan\backslash\Gamma$ containing $z$, which is an open ball with unique boundary point $\omega=\tau_\Gamma(z)$.
	We can consider the canonical retraction map $\tau_\Sigma\colon \overline{V}\to [z,\omega]$, where a point $x\in \overline{V}$ is retracted to $w_\Gamma(x,z)$ (cf.~Remark~\ref{Wegpunkt}).
Note that for $x\in \Sigma$, we have $$ g_{\zeta_0}(x,y)= \begin{cases}g_{\zeta_0}(\omega,\tau_\Gamma(y))_\Gamma & \text{ if } y\notin V,\\
g_{\zeta_0}(\omega,\tau_\Gamma(y))_\Gamma + \rho(w_{\Gamma}(x,y),\omega) & \text{ if } y\in V=\tau_\Sigma^{-1}([z,\omega)).
\end{cases}$$ Hence for  $x\in \Sigma$ the following is true
	\begin{align*}
	\int_{\Xan} g_{\zeta_0}(x,y)~d\nu(y)&= \int_{\Xan}g_{\zeta_0}(\omega,\tau_\Gamma(y))_\Gamma~d\nu(y)+ \int_{\tau_\Sigma^{-1}((\omega,z])}\rho(w_{\Gamma}(x,y),\omega)~d\nu(y)\\
	&=\int_{\Xan}g_{\zeta_0}(\omega,\tau_\Gamma(y))_\Gamma~d\nu(y)+ \int_{\tau_\Sigma^{-1}((\omega,z])}\rho(w_{\Gamma}(x,\tau_\Sigma(y)),\omega)~d\nu(y)\\
	&=\int_{\Xan}g_{\zeta_0}(\omega,\tau_\Gamma(y))_\Gamma~d\nu(y)+  \int_\Sigma \rho(w_{\Gamma}(x,t),\omega)~d((\tau_\Sigma)_*\nu)(t)\\
	&=\int_{\Xan}g_{\zeta_0}(\omega,\tau_\Gamma(y))_\Gamma~d\nu(y)+  \int_\Sigma g_\omega(x,t)_\Sigma~d((\tau_\Sigma)_*\nu)(t).
	\end{align*}
	Note that our path $\Sigma=[z,\omega]\subset \HM(\Xan)$ is a metric graph, and so we can consider the potential kernel  $g_\omega(x,t)_\Sigma$ on~$\Sigma$ from Definition~\ref{pkm}.
	For the last identity we used
    $ \rho(w_{\Gamma}(x,t),\omega)=  \rho(w_{\omega}(x,t),\omega)=g_\omega(x,t)_\Sigma$, which follows by Remark~\ref{VGL} and \cite[\S 4.2 p.~77]{BR}.
	Then Lemma~\ref{ggraph} tells us again that  the second term is finite and continuous. 
	As $g_{\zeta_0}(\omega,\tau_\Gamma(\cdot))_\Gamma=g_{\zeta_0}(\omega,\cdot)$  (see (\ref{Retraction})) is finitely valued and continuous on the compact set $\supp(\nu)$ by Proposition~\ref{G1}, the first one is a finite constant, and hence the claim follows.
\end{proof}

\begin{prop}\label{u1} Let $\zeta_0$ be  a chosen base point in $I(\Xan)$ and let $\nu$ be  a positive Radon measure on~$\Xan$. Then for  every  $\zeta\in I(\Xan)$ or $\zeta\notin \supp(\nu)$ the following are true:
	\begin{enumerate}
		\item If $\zeta\notin X(K)$, then $u_{\zeta_0,\nu}(\cdot,\zeta)$ is finitely valued and continuous on~$\Xan\backslash \supp(\nu)$ and it is  lsc on~$\Xan$.
		\item If $\zeta\in X(K)$, then $u_{\zeta_0,\nu}(\cdot,\zeta)$ is continuous on~$\Xan\backslash (\supp(\nu)\cup \{\zeta\})$  with\\ $u_{\zeta_0,\nu}(x,\zeta)=\infty$ if and only if $x=\zeta$, and it is lsc on~$\Xan\backslash \{\zeta\}$.
		\item For each $z\in \Xan$ and each path $[z,\omega]$, we have 
		\begin{align}\label{2.2}\liminf\limits_{t\to z}u_{\zeta_0,\nu}(t,\zeta)= \liminf\limits_{\substack{t\to z,\\ t\in I(\Xan)}}u_{\zeta_0,\nu}(t,\zeta)=\lim\limits_{\substack{t\to z,\\ t\in [\omega,z)}}u_{\zeta_0,\nu}(t,\zeta)= u_{\zeta_0,\nu}(z,\zeta).\end{align}
	\end{enumerate}
	Note that every probability measure on~$\Xan$  is a  positive Radon measure by Proposition~\ref{Radon}.
\end{prop}
\begin{proof}
	Recall from (\ref{C}) that we can write 
	\begin{align*}
	u_{\zeta_0,\nu}(\cdot,\zeta)=\int g_{\zeta_0}(\cdot,y)~d\nu(y)-\nu(\Xan)g_{\zeta_0}(\cdot,\zeta)+C_{\zeta_0,\zeta}.\end{align*}
	Since $g_{\zeta_0}(\cdot,\zeta)$ is finitely valued and continuous on~$\Xan$ if $\zeta\notin X(K)$ and $g_{\zeta_0}(\cdot,\zeta)$ is finitely valued and continuous on~$\Xan\backslash \{\zeta\}$ if $\zeta\in X(K)$ by  Proposition~\ref{G1}, it remains to show the assertions i) and ii) for the function $f(x):=\int g_{\zeta_0}(x,y)~d\nu(y)$ on~$\Xan$.
	As $g_{\zeta_0}$ is finitely valued and continuous off the diagonal by Lemma~\ref{g} and $\supp(\nu)$ is a compact subset, it follows that $f$ is finitely valued and continuous on~$\Xan\backslash \supp(\nu)$. 
	For the lower semi-continuity of $f$ we use techniques from the proof of \cite[Proposition~6.12]{BR}.
	By Lemma~\ref{g}, $g_{\zeta_0}$ is lower semi-continuous on the compact space $\Xan\times\Xan$, and so it is bounded from below by a constant $M$. 
	Using \cite[Proposition A.3]{BR}, we get the identity 
	$$f(x)=\sup \left\lbrace \int_{\Xan}g(x,y) ~d\nu(y) \mid g\in \CS^0(\Xan\times\Xan),~M \leq  g\leq g_{\zeta_0}\right\rbrace$$ on~$\Xan$.
	Due to the compactness of $\Xan$, the integral function $x\mapsto \int_{\Xan}g(x,y) ~d\nu(y)$ is continuous on~$\Xan$ for every $g\in \CS^0(\Xan\times\Xan)$.
	Then \cite[Lemma A.2]{BR} tells us that $f$ has  to be lower semi-continuous on~$\Xan$.

	Thus it remains to prove identity (\ref{2.2}).
	First, we show the last equation \begin{align*}\lim\limits_{t\to z, t\in [\omega,z)}u_{\zeta_0,\nu}(t,\zeta)= u_{\zeta_0,\nu}(z,\zeta).\end{align*}
	If $z\notin X(K)$, then by shrinking our path we may assume $[z,\omega]\subset \HM(\Xan)$,
and so the restriction of $u_{\zeta_0,\nu}(\cdot,\zeta) $ to $[z,\omega]$ is continuous by Lemma~\ref{u0} and the equation is true. 
	If $z\in X(K)$, we may assume that $[z,\omega]$ lies in a connected component of $\Xan\backslash \Gamma$ for a skeleton $\Gamma$ of $\Xan$ with $\zeta_0\in \Gamma$.
    Then $\tau_{\Gamma}(t)=\tau_\Gamma(z)$ for every $t\in (z,\omega]$, and so for every $y\in \Xan$ and  $t\in (z,\omega]$

    $$	g_{\zeta_0}(t,y)=\begin{cases}
    	g_{\zeta_0}(\tau_\Gamma(z),\tau_\Gamma(y))_\Gamma & \text{ if } \tau_\Gamma(z)\neq\tau_\Gamma(y),\\
    	g_{\zeta_0}(\tau_\Gamma(z),\tau_\Gamma(y))_\Gamma + \rho(w_\Gamma(t,y),\tau_\Gamma(z))  & \text{ if } \tau_\Gamma(z)=\tau_\Gamma(y).
    	\end{cases}$$ 
Since $\rho(w_{\Gamma}(t,y),\tau_\Gamma(z))$ increases monotonically as $t$ tends to $z$ along $(z,\omega]$ for every $y\in \Xan$,
the Monotone Convergence Theorem implies as in the proof of \cite[Proposition 6.12]{BR} that the integral function $\int g_{\zeta_0}(t,y)~d\nu(y)$ converges to $\int g_{\zeta_0}(z,y)~d\nu(y)$ as $t$ tends to $z$ along $(z,\omega]$. 
Furthermore, $g_{\zeta_0}(t,\zeta)$ converges to $g_{\zeta_0}(z,\zeta)$ as $t$ tends to $z$ along $(z,\omega]$ by Proposition~\ref{G1}.
At most one of the terms $\int g_{\zeta_0}(z,y)~d\nu(y)$ and $g_{\zeta_0}(z,\zeta)$ is infinite (due to $\zeta\in I(\Xan)$ or $\zeta\notin\supp(\nu)$), so the description stated at the beginning of the proof (or see (\ref{C})) implies 
\begin{align*}\lim\limits_{t\to z, t\in [\omega,z)}u_{\zeta_0,\nu}(t,\zeta)= u_{\zeta_0,\nu}(z,\zeta).\end{align*}

Now, we deduce the rest of (\ref{2.2}) from that.
When $\zeta=z\in X(K)$, we have $$\liminf\limits_{t\to z}u_{\zeta_0,\nu}(t,\zeta)\leq \liminf\limits_{\substack{t\to z,\\ t\in I(\Xan)}}u_{\zeta_0,\nu}(t,\zeta)\leq \lim\limits_{\substack{t\to z,\\ t\in [\omega,z)}}u_{\zeta_0,\nu}(t,\zeta)= u_{\zeta_0,\nu}(z,\zeta)=-\infty,$$
and so clearly (\ref{2.2}) is true.
When $\zeta\notin X(K)$ or $\zeta\neq z$, then $u_{\zeta_0,\nu}( \cdot,\zeta) $ is lsc at $z$ by i) and ii), and so we get 
	$$u_{\zeta_0,\nu}(z,\zeta) \leq  \liminf\limits_{t\to z}u_{\zeta_0,\nu}(t,\zeta) \leq  \liminf\limits_{\substack{t\to z,\\ t\in I(\Xan)}}u_{\zeta_0,\nu}(t,\zeta) \leq  \lim\limits_{\substack{t\to z,\\ t\in [\omega,z)}}u_{\zeta_0,\nu}(t,\zeta) =   u_{\zeta_0,\nu}(z,\zeta) .$$
	Hence we also have equality.
\end{proof}

\begin{prop}\label{u2}Let $\zeta_0$ be  a chosen base point in $I(\Xan)$ and  let $\nu$ be  any finite signed Borel measure on~$\Xan$.
	Then  for every  $\zeta\in I(\Xan)$ or $\zeta\notin \supp(\nu)$, the potential function $u_{\zeta_0,\nu}(\cdot,\zeta)$  defines a current in $D^0(\Xan)$ with 
	$$dd^c u_{\zeta_0,\nu}(\cdot,\zeta)= \nu(\Xan)\delta_{\zeta}-\nu.$$
\end{prop}
\begin{proof}
A function on~$\Xan$ defines a current in $D^0(\Xan)$ if and only if its restriction to $ I(\Xan)$ is finite (cf.~Proposition~\ref{Hom}). 
Recall from $(\ref{C})$ that for every $x\in \Xan$	\begin{align*}
	u_{\zeta_0,\nu}(\cdot,\zeta)  =\int g_{\zeta_0}(\cdot,y)~d\nu(y)-\nu(\Xan)g_{\zeta_0}(\cdot,\zeta)+C_{\zeta_0,\zeta}\nonumber.\end{align*}
If we fix $x\in I(\Xan)$, the function $ g_{\zeta_0}(x,\cdot)=g_{\zeta_0}(\cdot,x)$ (symmetry follows by Lemma~\ref{g}) is a finitely valued continuous function on~$\Xan$ by Proposition~\ref{G1} i).
Hence all terms define currents in $D^0(\Xan)$, and so does $u_{\zeta_0,\nu}(\cdot,\zeta)$.
For the first term we also use that $\supp(\nu)$ is compact. 
Furthermore, we know by Proposition~\ref{G1}  that for any fixed $y$ we have 
$ dd^c g_{\zeta_0}(\cdot,y)=\delta_{\zeta_0}-\delta_y$.
	Due to the calculation
	\begin{align*}
		\langle  dd^c \left(\int  g_{\zeta_0}(\cdot,y)~d\nu(y)\right),\varphi\rangle& = \int \langle dd^c  g_{\zeta_0}(\cdot,y),\varphi\rangle ~d\nu(y)\\
		&= \int \left(\int \varphi ~d(\delta_{\zeta_0}-\delta_y)(x)\right)d\nu(y)\\&= \int \varphi ~d(\nu(\Xan)\delta_{\zeta_0}-\nu)(y)
	\end{align*}
	for every $\varphi\in A_c^0(\Xan)$, we obtain 
$$ dd^c \left(\int g_{\zeta_0}(\cdot,y)~d\nu(y)\right) =\nu(\Xan)\delta_{\zeta_0} -\nu.$$
Hence
	$$ dd^c u_{\zeta_0,\nu}(\cdot,\zeta)=   \nu(\Xan)\delta_{\zeta_0}-\nu - \nu(\Xan)(\delta_{\zeta_0}-\delta_\zeta)=\nu(\Xan)\delta_\zeta-\nu. $$
\end{proof}

\section{Arakelov--Green's functions}\label{Arakelov}
Baker and Rumely developed a theory of Arakelov--Green's functions on~$\mathbb{P}^{1,\an}$ in  \cite[\S 8.10]{BR}.
This class of functions  arise naturally in the study of dynamics and can be seen as  a generalization of the potential kernel from Section~\ref{Potentialkernel}. 
Arakelov--Green's functions are characterized by a list of properties which can be found in Definition~\ref{AGlist}.
We generalize Baker and Rumely's definition of an Arakelov--Green's function from $\mathbb{P}^{1,\an}$ to $\Xan$, and show that the characteristic properties are still satisfied.

\begin{defn}\label{AGlist}
A symmetric function $g$ on~$\Xan\times \Xan$ that satisfies the following list of properties for a probability measure $\mu$ on~$\Xan$ is called a \emph{normalized Arakelov--Green's function} on~$\Xan$.
\begin{enumerate}
\item(Semicontinuity) The function $g$  is finite and continuous off the diagonal and strongly lower semi-continuous on the diagonal in the sense that 
$$g(x_0,x_0)= \liminf _{(x,y)\to (x_0,x_0) ,x\neq y}g(x,y).$$
\item (Differential equation) For each fixed $y\in \Xan$ the function  $g(\cdot,y)$ is an element of $D^0(\Xan)$
and  $$ dd^c g(\cdot,y)=\mu - \delta_y.$$
\item(Normalization)
$$ \int\int g(x,y)~d\mu(x)d\mu(y)=0.$$\end{enumerate}
\end{defn}
 The list of properties is an analog of the one in the complex case and can for example also be found in 
 \cite[\S 3.5 (B1)-(B3)]{BR2}.
 
\begin{bem}\label{AGunique} As in the complex case, the list of properties in  Definition~\ref{AGlist} for a probability measure $\mu$ on~$\Xan$ determines a normalized Arakelov--Green's function on~$\Xan$ uniquely.
If $\widetilde{g}$ is another symmetric function on~$\Xan\times\Xan$ satisfying i)-iii), then for a fixed $y\in \Xan$ $$g(\cdot,y)-\widetilde{g}(\cdot,y)=h_y$$ on~$I(\Xan)$ for a harmonic function $h_y$ on~$\Xan$ 
by property ii) and \cite[Lemme 3.3.12]{Th}.
This harmonic function $h_y$ has to be constant on~$\Xan$  by the Maximum Principle (Proposition~\ref{Max}).
Since $I(\Xan)$ is dense in $\Xan$,  the identity holds on all of $\Xan$ by property i).

Thanks to the symmetry of $g$ and $\widetilde{g}$, the constant function $h_y$ is independent of $y$.
The last property iii), implies that this constant has to be zero, i.e.~$g=\widetilde{g}$ on~$\Xan\times\Xan$.
\end{bem}

\begin{defn}\label{Dlog-c}
A probability measure $\mu$ on~$\Xan$  \emph{has continuous potentials} if  each $\zeta\in I(\Xan)$ defines a continuous function
$$\Xan\to \R,~x\mapsto \int_{\Xan} g_{\zeta}(x,y)~d\mu(y).$$
These functions are bounded as $\Xan$ is compact.
\end{defn}
\begin{bem}
	Let $\mu$ be a probability measure on~$\Xan$. If there exists a point $\zeta_0\in I(\Xan)$ such that 
	$\Xan\to \R,~x\mapsto \int_{\Xan} g_{\zeta_0}(x,y)~d\mu(y)$
	defines a continuous function, then	$\mu$ has continuous potentials.
\end{bem}

\begin{Ex}\label{log-c}
Let $\mu$ be a probability measure supported on a skeleton $\Gamma$ of $\Xan$ (e.g.~$\mu=\delta_z$ for some $z\in I(\Xan)$), then $\mu$ has continuous potentials (using the last remark and Lemma~\ref{ggraph}).
\end{Ex}

\begin{defn}\label{D Arakelov} For every probability measure $\mu$ on~$\Xan$ with continuous potentials and a fixed base point $\zeta_0\in I(\Xan)$, we define  $ g_{\zeta_0,\mu}\colon \Xan\times\Xan \to (-\infty,\infty]$ by
$$ g_{\zeta_0,\mu}(x,y):=g_{\zeta_0}(x,y)- \int_{\Xan} g_{\zeta_0}(x,\zeta)~d\mu(\zeta)-\int_{\Xan}g_{\zeta_0}(y,\zeta)~d\mu(\zeta)+C_{\zeta_0},$$  where $C_{\zeta_0}$ is a constant chosen such that $$\int\int  g_{\zeta_0,\mu}(x,y)~d\mu(x)d\mu(y)=0.$$\end{defn}
\begin{bem}
Recall that 
$g_{\zeta_0}(\zeta_0,x,y)= g_{\zeta_0}(x,y)$ (see Definition~\ref{Kernel}) by Corollary~\ref{3Var},
and so the potential function from Section~\ref{Potentialfunctions} can be written as
$$u_{\zeta_0,\mu}(\cdot,\zeta_0)=\int g_{\zeta_0}(\zeta_0,\cdot,\zeta)~ d\mu(\zeta)=\int g_{\zeta_0}(\cdot,\zeta)~ d\mu(\zeta).$$
Hence we have the description
\begin{align}\label{a1}
g_{\zeta_0,\mu}(x,y)= g_{\zeta_0}(x,y)- u_{\zeta_0,\mu}(x,\zeta_0) -u_{\zeta_0,\mu}(y,\zeta_0)+ C_{\zeta_0}
\end{align} on~$\Xan\times \Xan$.
\end{bem}

In the following lemma, we see that this function is independent of the chosen base point, and hence we just write $g_{\mu}$.
\begin{lem}\label{L-ind}
 For every probability measure $\mu$ on~$\Xan$ with continuous potentials, the function $ g_{\zeta_0,\mu}$ is independent of the chosen base point $\zeta_0$.
\end{lem}
\begin{proof} First, we determine $C_{\zeta_0}$
\begin{align*}
0=\int\int  g_{\zeta_0,\mu}(x,y)~d\mu(x)d\mu(y)
&=  \int\int g_{\zeta_0}(x,y)~d\mu(x)d\mu(y)-\int\int g_{\zeta_0}(x,\zeta)~d\mu(\zeta)d\mu(x)\\
&~~~-\int \int g_{\zeta_0}(y,\zeta)~d\mu(\zeta)d\mu(y) + C_{\zeta_0}\\
&= -\int\int g_{\zeta_0}(x,y)~d\mu(x)d\mu(y) + C_{\zeta_0}.
\end{align*}
Hence $C_{\zeta_0}=\int\int g_{\zeta_0}(x,y)~d\mu(x)d\mu(y)$. 
Now let $\zeta_0'\in I(\Xan)$. Applying  Lemma~\ref{g} to  $C_{\zeta_0}$, we get 
\begin{align*}
	C_{\zeta_0}&= \int\int g_{\zeta_0}(x,y)~d\mu(x)d\mu(y)\\
	&= \int\int \left( g_{\zeta'_0}(x,y)-g_{\zeta'_0}(x,\zeta_0)-g_{\zeta'_0}(y,\zeta_0)+g_{\zeta'_0}(\zeta_0,\zeta_0)\right)~d\mu(x)d\mu(y)\\
	&= C_{\zeta'_0} - 2 \int g_{\zeta'_0}(x,\zeta_0)~d\mu(x)+g_{\zeta'_0}(\zeta_0,\zeta_0),
\end{align*}
 where $-2\int g_{\zeta'_0}(\zeta,\zeta_0)~d\mu(\zeta)+g_{\zeta'_0}(\zeta_0,\zeta_0) $ is a finite constant as $\mu$ has continuous potentials.

Using  Lemma~\ref{g} also for the other terms of $ g_{\zeta_0,\mu}$, i.e.~for $g_{\zeta_0}(x,y)$, $g_{\zeta_0}(x,\zeta)$ and $g_{\zeta_0}(y,\zeta)$, and plugging in the identity from above, we get
$ g_{\zeta_0,\mu}(x,y)= g_{\zeta'_0,\mu}(x,y).$
\end{proof}
\begin{prop}\label{G}
 Let $\mu$ be a probability measure on~$\Xan$ with continuous potentials. Then as a function of two variables $g_{\mu}\colon \Xan\times \Xan\to (-\infty,\infty]$ is symmetric, finite and continuous off the diagonal,  and strongly lower semi-continuous on the diagonal in the sense that 
$$g_{\mu}(x_0,x_0)= \liminf _{(x,y)\to (x_0,x_0) ,x\neq y}g_{\mu}(x,y),$$
where we understand $\Xan\times \Xan$  set theoretically and endowed with the product topology.
\end{prop}
\begin{proof} As
$$g_{\mu}(x,y)= g_{\zeta_0}(x,y)-\int g_{\zeta_0}(x,\zeta)~ d\mu(\zeta) -\int g_{\zeta_0}(y,\zeta)~ d\mu(\zeta)+ C_{\zeta_0}$$ for some base point $\zeta_0\in I(\Xan)$ and as we required $\mu$ to has continuous potentials, 
Lemma~\ref{g} implies that $g_{\mu}\colon \Xan\times \Xan\to (-\infty,\infty]$  is symmetric, finite and continuous off the diagonal and lsc on~$\Xan\times\Xan$.
Thus we only need to prove
\begin{align}\label{Ungl}g_{\mu}(x_0,x_0)\geq \liminf _{(x,y)\to (x_0,x_0) ,x\neq y}g_{\mu}(x,y)=\sup_{U\in \US((x_0,x_0))}\inf_{(x,y)\in U\backslash (x_0,x_0)} g_{\mu}(x,y) .\end{align}
Here $\US((x_0,x_0))$ is any basis of open neighborhoods of $(x_0,x_0)$ in $\Xan\times \Xan$ endowed with the product topology.

  In the following, 
let $\Gamma$ be any skeleton of $\Xan$ with $\zeta_0\in \Gamma$.
If $x_0$ is of type I, we have $g_{\mu}(x_0,x_0)=g_{\zeta_0}(x_0,x_0)=\infty$ by the definition of the potential kernel, and so $(\ref{Ungl})$ is obviously true. 

If $x_0$ is of type II or III, we may choose $\zeta_0=x_0$ by Lemma~\ref{L-ind}, and so
$$ g_\mu(x_0,x_0)=g_{x_0}(x_0,x_0)-\int g_{x_0}(x_0,\zeta)~ d\mu(\zeta) -\int g_{x_0}(x_0,\zeta)~ d\mu(\zeta)+ C_{\zeta_0}=C_{\zeta_0}$$
as  $g_{x_0}(x_0,\zeta)=0$ for every $\zeta\in \Xan$ by Lemma~\ref{g}.
On the other hand, every $U$ in $\US((x_0,x_0))$ contains an element of the form $(x_0,y)$ with $y\in \Xan\backslash \{x_0\}$, and  
\begin{align*}
	g_\mu(x_0,y)& =g_{x_0}(x_0,y)-\int g_{x_0}(x_0,\zeta)~ d\mu(\zeta) -\int g_{x_0}(y,\zeta)~ d\mu(\zeta)+ C_{\zeta_0}\\ & 
	= -\int g_{x_0}(y,\zeta)~ d\mu(\zeta)+ C_{\zeta_0} \leq C_{\zeta_0}
\end{align*}
since $\mu$ and $g_{x_0}(y,\cdot)$ are non-negative (see Lemma~\ref{g} i)).
Thus (\ref{Ungl}) has to be true.

For the rest of the proof let $x_0$ be of type IV.
 There is a basis of open neighborhoods of $x_0$ that is contained in the connected component $V$ of $\Xan\backslash \Gamma$ that contains $x_0$ (cf.~Theorem~\ref{Str}).
Consider the corresponding basis of open neighborhoods $\US((x_0,x_0))$ of $(x_0,x_0)$ in $\Xan\times \Xan$ endowed with the product topology. 
In every  $U\in \US((x_0,x_0))$ we consider tuples of the form $(x_0,y)$ where $y$ lies in the interior of the unique path $[x_0,\tau_\Gamma(x_0)]$ (such tuples always exist). 
Then $\tau_\Gamma(y)=\tau_\Gamma(x_0)$ and $w_{\Gamma}(x_0,y)=y$ (recall its definition from Remark~\ref{Wegpunkt}), and so
\begin{align*}
g_{\zeta_0}(x_0,x_0)-g_{\zeta_0}(x_0,y)&= g_{\zeta_0}(\tau_\Gamma(x_0),\tau_\Gamma(x_0))_\Gamma+\rho(w_{\Gamma}(x_0,x_0),\tau_\Gamma(x_0))\\ 
&~~~-(g_{\zeta_0}(\tau_\Gamma(x_0),\tau_\Gamma(y))_\Gamma+\rho(w_{\Gamma}(x_0,y),\tau_\Gamma(y)))\\
&=  \rho(w_{\Gamma}(x_0,x_0),\tau_\Gamma(x_0))-\rho(w_{\Gamma}(x_0,y),\tau_\Gamma(y))\\
&= \rho(x_0,\tau_\Gamma(x_0))-\rho(y,\tau_\Gamma(x_0))\\
&=  \rho(x_0,y).
\end{align*}
Consequently, we get 
\begin{align}\label{Int1}
g_{\mu}(x_0,x_0)- g_{\mu}(x_0,y)& = g_{\zeta_0}(x_0,x_0)-2\int g_{\zeta_0}(x_0,\zeta)~ d\mu(\zeta) + C_{\zeta_0} \nonumber\\&~~ -(g_{\zeta_0}(x_0,y)-\int g_{\zeta_0}(x_0,\zeta)~ d\mu(\zeta)-\int g_{\zeta_0}(y,\zeta)~ d\mu(\zeta) + C_{\zeta_0} )  \nonumber \\
& = g_{\zeta_0}(x_0,x_0)-g_{\zeta_0}(x_0,y)-\int g_{\zeta_0}(x_0,\zeta)~ d\mu(\zeta) +\int g_{\zeta_0}(y,\zeta)~ d\mu(\zeta)\nonumber \\& = 
\rho(x_0,y)+ \int g_{\zeta_0}(y,\zeta)-g_{\zeta_0}(x_0,\zeta)~ d\mu(\zeta).
\end{align}
To prove (\ref{Ungl}), we need to show that (\ref{Int1}) is non-negative. 

Recall that $y$ lies in the interior of the unique path $[x_0,\tau_\Gamma(x_0)]$.
We denote by $V_0$ the connected component of $V\backslash \{y\}$ that contains $x_0$ (note that $V_0$ is an open ball as $x_0$ is of type IV and $V$ is an open ball).
 We will see that $g_{\zeta_0}(y,\zeta)-g_{\zeta_0}(x_0,\zeta)$ in (\ref{Int1}) is zero for every $\zeta\in \Xan\backslash V_0$.
Recall that $V$ is the connected component of $\Xan\backslash \Gamma$ that contains $x_0$. 
Hence $V$ is an open ball with $\partial V=\{\tau_\Gamma(x_0)\}$ and $V_0\subset V$. 
Furthermore, one should have in mind that $\tau_\Gamma(y)=\tau_\Gamma(x_0)$ as $y\in [x_0,\tau_\Gamma(x_0)]$.

If $\zeta\in \Xan \backslash V$, then by the definition of the potential kernel
$$ g_{\zeta_0}(y,\zeta)-g_{\zeta_0}(x_0,\zeta)=g_{\zeta_0}(\tau_\Gamma(y),\tau_\Gamma(\zeta))_\Gamma- g_{\zeta_0}(\tau_\Gamma(x_0),\tau_\Gamma(\zeta))_\Gamma=0.$$ 

If $\zeta\in V\backslash V_0$, then $\tau_\Gamma(\zeta)=\tau_\Gamma(x_0)=\tau_\Gamma(y)$  and $w_{\Gamma}(x_0,\zeta)=w_{\Gamma}(y,\zeta)$, and hence
 \begin{align*}
g_{\zeta_0}(y,\zeta)-g_{\zeta_0}(x_0,\zeta)&=g_{\zeta_0}(\tau_\Gamma(y),\tau_\Gamma(\zeta))_\Gamma+\rho(w_{\Gamma}(y,\zeta),\tau_\Gamma (y))\\ &~~~-( g_{\zeta_0}(\tau_\Gamma(x_0),\tau_\Gamma(\zeta))_\Gamma+\rho(w_{\Gamma}(x_0,\zeta),\tau_\Gamma (x_0)))\\&=0.
\end{align*}
Thus $g_{\zeta_0}(y,\zeta)-g_{\zeta_0}(x_0,\zeta)=0$ for every $\zeta\in \Xan\backslash V_0$.

For every $\zeta\in V_0$ we have $\tau_\Gamma(\zeta)=\tau_\Gamma(x_0)=\tau_\Gamma(y)$, $w_{\Gamma}(y,\zeta)=y$, and $w_{\Gamma}(x_0,\zeta)\in [x_0,y]$.
Hence
\begin{align*}
g_{\zeta_0}(y,\zeta) -g_{\zeta_0}(x_0,\zeta)&=g_{\zeta_0}(\tau_\Gamma(y),\tau_\Gamma(\zeta))_\Gamma+\rho(w_{\Gamma}(y,\zeta),\tau_\Gamma (y))\\ &~~~-( g_{\zeta_0}(\tau_\Gamma(x_0),\tau_\Gamma(\zeta))_\Gamma+\rho(w_{\Gamma}(x_0,\zeta),\tau_\Gamma (x_0)))\\
&= \rho( w_{\Gamma}(y,\zeta),\tau_\Gamma(y)) - \rho( w_{\Gamma}(x_0,\zeta),\tau_\Gamma(x_0))\\
& = \rho( y,\tau_\Gamma(x_0)) - \rho( w_{\Gamma}(x_0,\zeta),\tau_\Gamma(x_0))\\
&= - \rho( w_{\Gamma}(x_0,\zeta),y) 
\end{align*} for every $\zeta\in V_0$.
Plugging everything in (\ref{Int1}), we get
\begin{align*}
g_{\mu}(x_0,x_0)- g_{\mu}(x_0,y)&= \rho(x_0,y)+ \int_{V_0}- \rho( w_{\Gamma}(x_0,\zeta),y)~ d\mu(\zeta)\\& \geq   \int_{V_0}\rho(x_0,y)- \rho( w_{\Gamma}(x_0,\zeta),y)~ d\mu(\zeta)\\& \geq 0
\end{align*}
as $\rho( x_0,y)\geq  \rho( w_{\Gamma}(x_0,\zeta),y)$ on~$V_0$ and $\mu$ is a non-negative measure.

Consequently, $(\ref{Ungl})$ has to be also true for $x_0$ of type IV.
\end{proof}

\begin{prop}\label{G'}
 For every probability measure $\mu$ on~$\Xan$ with continuous potentials and for every fixed $y\in \Xan$, the function $G_{\mu,y}:=g_{\mu}(\cdot,y)\colon \Xan\to 
(-\infty,\infty]$ defines a current in $D^0(\Xan)$ and satisfies 
$$dd^c G_{\mu,y}= \mu -\delta_y.$$
Moreover, $G_{\mu,y}$ is continuous on~$\Xan$ with $G_{\mu,y}(x)=\infty$ if and only if $x=y\in X(K)$. 
In particular, $G_{\mu,y}$ is subharmonic on~$\Xan\backslash\{y\}$.
\end{prop}
\begin{proof}
By the definition of the Arakelov--Green's function, we have 
\begin{align*}
G_{\mu,y}(x)&= g_{\zeta_0}(x,y)-\int g_{\zeta_0}(x,\zeta)~ d\mu(\zeta) -\int g_{\zeta_0}(y,\zeta)~ d\mu(\zeta)+ C_{\zeta_0}
\end{align*} for every $x\in \Xan$.
 Due to Proposition~\ref{G1}, the first term is continuous on~$\Xan$ and attains values in $\R\cup\{\infty\}$ with $g_{\zeta_0}(x,y)=\infty$ if and only if $x=y\in X(K)$. 
In particular, the first term is finitely valued on~$I(\Xan)$.
 Since  $\mu$ has continuous potentials, the other two terms are finitely valued and continuous on~$\Xan$.
Hence  $G_{\mu,y}\colon \Xan\to (-\infty,\infty]$ is continuous on~$\Xan$  with  $G_{\mu,y}(x)=\infty$ if and only if $x=y\in X(K)$. In particular, $G_{\mu,y}$ is finitely valued on~$I(\Xan)$, and so 
defines a current in $D^0(\Xan)$ by Proposition~\ref{Hom}.
It remains to calculate the Laplacian of $G_{\mu,y}$. 
By (\ref{a1}), we have
\begin{align*}
G_{\mu,y}&= g_{\zeta_0}(\cdot,y)- u_{\zeta_0,\mu}(\cdot,\zeta_0) -u_{\zeta_0,\mu}(y,\zeta_0)+ C_{\zeta_0}.
\end{align*} 
 Proposition~\ref{G1} and Proposition~\ref{u2} imply
$dd^c G_{\mu,y} =\mu-\delta_y.$
Hence $G_{\mu,y}$ is subharmonic on~$\Xan\backslash \{y\}$.
\end{proof}

Using the previous propositions,  $g_{\mu}$ is  a normalized Arakelov--Green's function as defined in Definition~\ref{AGlist} for every  probability measure $\mu$ on~$\Xan$ with continuous potentials.

\begin{kor}\label{AG}
	Let $\mu$ be a probability measure on~$\Xan$ with continuous potentials. Then the function $g_{\mu}$ is a normalized Arakelov--Green's function on~$\Xan$.
\end{kor}
\begin{proof}
	We need to know that all properties of the list in Definition~\ref{AGlist} hold.
	Property i) and symmetry are true due to Proposition~\ref{G}, ii) was shown in Proposition~\ref{G'}, and iii) follows by construction.
\end{proof}

\begin{bem}
	 For a  probability measure $\mu$ on~$\Xan$ and for a  point $y\in I(\Xan)$, Thuillier constructs in his thesis \cite[\S 3.4.3]{Th} a unique function $g_{y,\mu}\colon \Xan \to [-\infty,\infty)$ such that $dd^c g_{y,\mu}=\mu -\delta_y$, $g_{y,\mu}(y)=0$ and its restriction to $\Xan\backslash \{y\}$ is subharmonic.
	 His construction uses  \cite[Th\'eor\`eme 3.3.13 \& 3.4.12]{Th}.
	 If $\mu$ has continuous potentials, then $g_{y,\mu}$ and  $G_{\mu,y}$ define two currents in $D^0(\Xan)$ (cf.~Proposition~\ref{Prop not lisse}) having the same Laplacian $\mu-\delta_{y}$.
	 \cite[Lemma 3.3.12]{Th} implies that $g_{y,\mu}$ and  $G_{\mu,y}$  differ only by a harmonic function on~$\Xan$, which has to be constant by the Maximum Principle \ref{Max}.
\end{bem}

\section{Energy Minimization Principle}\label{EMP}
The Energy Minimization Principle is a very important theorem in dynamics and has many applications. 
The goal is to translate this principle into our non-archimedean setting.
For $X=\PB^{1}$ this was already done in \cite[\S 8.10]{BR}, and Matt Baker suggested to generalize their definition of Arakelov--Green's functions and their result to the author. 
In the following section, we give a proof of the Energy Minimization Principle for a smooth projective curve $X$ over our non-archimedean field $K$ using the techniques from \cite[\S 8.10]{BR}.

\begin{defn} Let $\mu$ be a probability measure on~$\Xan$ with continuous potentials. Then for every probability measure $\nu$ on~$\Xan$, we define the  corresponding $\mu$-\emph{energy integral} as 
	$$I_\mu(\nu):=\int \int  g_{\mu}(x,y) ~d\nu (y)d\nu(x).$$
Note that the integral is well-defined since  $g_{\mu} $ is lower semi-continuous on $\Xan\times \Xan$ by Proposition~\ref{G}, and hence Borel measurable.
\end{defn}

\begin{Thm}[Energy Minimization Principle]\label{THM}
	Let $\mu$ be a probability measure on~$\Xan$ with continuous potentials. Then 
	\begin{enumerate}
		\item $I_\mu(\nu)\geq 0$ for each probability measure $\nu$ on~$\Xan$, and 
		\item $I_\mu(\nu)=0$ if and only if $\nu=\mu$.
	\end{enumerate}
\end{Thm}
We show the principle in several steps. 
At first, we prove analogues of Maria's theorem (Theorem~\ref{Maria}) and Frostman's theorem (Theorem~\ref{Frostman}). 
In Maria's theorem we study the boundedness of the generalized potential function that is defined in the subsequent definition.

\begin{defn} Let $\mu$ be a probability measure on~$\Xan$ with continuous potentials. Then for every probability measure $\nu$ on~$\Xan$, we define the corresponding \emph{generalized potential function} by 
	$$u_\nu(\cdot,\mu):=\int  g_{\mu}(\cdot,y) ~d\nu (y).$$
\end{defn}

\begin{lem}\label{uuu}
Let $\mu$ be a  probability measure with continuous potentials and let $\nu$ be an arbitrary probability measure on~$\Xan$. 
Then for every $\zeta_0\in I(\Xan)$ we can write 
\begin{align}\label{uu}
u_\nu(\cdot,\mu)=u_{\zeta_0,\nu}(\cdot,\zeta_0)- u_{\zeta_0,\mu}(\cdot,\zeta_0) +C
\end{align}
on $\Xan$ for a finite constant $C$.
\end{lem}
\begin{proof}
Let $\zeta_0$ be a point in $I(\Xan)$. 
Then  by Corollary~\ref{3Var} $$u_{\zeta_0,\nu}(\cdot,\zeta_0)=\int g_{\zeta_0}(\zeta_0,\cdot,\zeta)~ d\nu(\zeta)=\int g_{\zeta_0}(\cdot,\zeta)~ d\nu(\zeta).$$
The same identity is true for $\mu$, i.e.~$u_{\zeta_0,\mu}(\cdot,\zeta_0)=\int g_{\zeta_0}(\cdot,\zeta)~ d\mu(\zeta)$,
which  is  a finitely valued continuous function on~$\Xan$ as $\mu$ has continuous potentials. 
	Thus we can write using the definition of the Arakelov--Green's function (Definition~\ref{D Arakelov}) 
	\begin{align*}
	u_\nu(x,\mu)&= \int g_\mu(x,y)~d\nu(y) \\
 &= \int g_{\zeta_0}(x,y)~d\nu(y)-\int g_{\zeta_0}(x,\zeta)~ d\mu(\zeta) -\int\int  g_{\zeta_0}(y,\zeta)~ d\mu(\zeta) d\nu(y) + C_{\zeta_0}\\
 &=u_{\zeta_0,\nu}(x,\zeta_0)- u_{\zeta_0,\mu}(x,\zeta_0)-\int u_{\zeta_0,\mu}(y,\zeta_0)~d\nu(y) + C_{\zeta_0}
	\end{align*}
for every $x\in \Xan$.
Since $u_{\zeta_0,\mu}(\cdot,\zeta_0)$ is bounded and continuous on~$\Xan$, we get 
$$u_\nu(\cdot,\mu)=u_{\zeta_0,\nu}(\cdot,\zeta_0)- u_{\zeta_0,\mu}(\cdot,\zeta_0) +C$$ on~$\Xan$ for a finite constant $C$.
\end{proof}
\begin{prop}\label{U}Let $\mu$ be a probability measure with continuous potentials and let $\nu$ be an arbitrary probability measure on~$\Xan$. Then
 $u_\nu(\cdot,\mu)\colon \Xan\to (-\infty,\infty]$ is continuous on~$\Xan\backslash \supp(\nu)$ and lsc on~$\Xan$.
Moreover, the restriction of  $u_\nu(\cdot,\mu)$ to every skeleton $\Gamma$ of $\Xan$ and to every path $[y,z]$  is finite and continuous.
\end{prop}
\begin{proof}
Let $\zeta_0$ be some point in $I(\Xan)$, then  
$$u_\nu(\cdot,\mu)=u_{\zeta_0,\nu}(\cdot,\zeta_0)- u_{\zeta_0,\mu}(\cdot,\zeta_0) +C$$ on~$\Xan$ for a finite constant $C$ by Lemma~\ref{uuu}.
Since $\mu$ has continuous potentials, $u_{\zeta_0,\mu}(\cdot,\zeta_0)$ is a finitely valued continuous function on $\Xan$.
Thus  it remains to prove  the continuity assertions for $u_{\zeta_0,\nu}(\cdot,\zeta_0)$. But these were all already shown in Lemma~\ref{u0} and Proposition~\ref{u1}. 
\end{proof}

\begin{prop}\label{UU}Let $\mu$ be a probability measure with continuous potentials and let $\nu$ be an arbitrary probability measure on~$\Xan$. Then
$u_\nu(\cdot,\mu)$ defines a current in $D^0(\Xan)$ with
	$$dd^c u_\nu(\cdot,\mu) = \mu-\nu.$$
	In particular, $u_\nu(\cdot,\mu)$ is subharmonic on~$\Xan\backslash\supp(\nu)$.
\end{prop}
	 
\begin{proof}
Let $\zeta_0$ be a point in $I(\Xan)$, then  
$$u_\nu(\cdot,\mu)=u_{\zeta_0,\nu}(\cdot,\zeta_0)- u_{\zeta_0,\mu}(\cdot,\zeta_0) +C$$ on~$\Xan$ for a finite constant $C$ by Lemma~\ref{uuu}.
By Proposition~\ref{u2} and linearity, the function  $u_{\zeta_0,\mu}(\cdot,\zeta_0) $
 belong to $D^0(\Xan)$ with $dd^c u_\nu(\cdot,\mu)= \mu-\nu.$
Then the generalized potential function $u_\nu(\cdot,\mu)$ is therefore subharmonic on~$\Xan\backslash \supp(\nu)$ as it is upper semi-continuous by  Proposition~\ref{U}.
\end{proof}

The key tool of the proof of Maria's theorem in \cite{BR} is \cite[Proposition 8.16]{BR}, which we can translate to our situation in the following form.

\begin{lem}\label{path}
	Let $W$ be an open ball or an open annulus in $\Xan$ and let $f$ be a subharmonic function on a connected open subset  $V$ of $W$ with $\overline{V}\subset W$. 
	For every $x\in \HM(V)$, there is a path $\Lambda$ from $x$ to a boundary point $y\in \partial V$ such that $f$ is non-decreasing along $\Lambda$.
\end{lem}
\begin{proof}
Since $\overline{V}$ is contained in an open ball or in an open annulus, we can view it is a subset of $\mathbb{P}^{1,\an}$. 
Then \cite[Proposition 8.16]{BR} and Remark \ref{BemSH} yield the claim.
\end{proof}
With the help of Proposition~\ref{U} and Lemma~\ref{path}, we can prove Maria's theorem. 
\begin{Thm}[Maria]\label{Maria}
	Let $\mu$ be a probability measure on~$\Xan$ with continuous potentials and let $\nu$ be an arbitrary probability measure on~$\Xan$. If there is a constant $M<\infty$ such that $u_\nu(\cdot,\mu)\leq M$ on~$\supp(\nu)$, then $u_\nu(\cdot,\mu)\leq M$ on~$\Xan$.
\end{Thm}
\begin{proof}
	Let $V$ be a connected component of $\Xan\backslash \supp(\nu)$ and assume there is a point $x_0\in V$ such that $u_\nu(x_0,\mu)> M$.
	Note that $V$ is path-connected since $\Xan$ is locally path-connected. 
If $B$ is an open ball in $\Xan$, then between two points $x,y\in B$ there is only one path in $\Xan$ by the structure of $\Xan$.
Thus $V\cap B$  is uniquely path-connected for every open ball $B$ in $\Xan$.
We have seen in Proposition~\ref{U} that the generalized potential function $u_\nu(\cdot,\mu)$ is continuous on~$V\subset\Xan\backslash \supp(\nu)$.
Hence we may assume $x_0$ to be contained in the dense subset $I(V)$ of $V$, and so we can choose a skeleton $\Gamma$ of $\Xan$ containing $x_0$ by Proposition~\ref{skeletonprop}.

Let $(Y_\alpha)_\alpha$ be the directed system of connected strictly affinoid domains contained in $V$ and containing $x_0$.
Note that the union of two connected strictly affinoid domains $Y_1,Y_2$ in $\Xan$ both containing $x_0$  with $ Y_1\cup Y_2\neq \Xan$    is again a connected strictly affinoid domain in $\Xan$ by \cite[Corollaire 2.1.17]{Th}.
Then $u_\nu(\cdot,\mu)$ is continuous on~$Y_\alpha$ and subharmonic on the relative interior $Y_\alpha^\circ$ by Proposition~\ref{U} and Proposition~\ref{UU}. Hence $u_\nu(\cdot,\mu)$  attains a maximum on~$Y_\alpha$ in a point $z_\alpha\in \partial Y_\alpha$ (see Maximum Principle~\ref{Max}), i.e.
\begin{align}\label{STRU}
u_\nu(z_\alpha,\mu)=\max_{x\in Y_\alpha} u_\nu(x,\mu)\geq \max_{x\in Y_\alpha^\circ} u_\nu(x,\mu)\geq u_\nu(x_0,\mu)>M
\end{align} for every $\alpha$.
Then $\langle z_\alpha\rangle _\alpha$ defines a net of type II points in $V$.
As $\overline{V}$ is compact,  we may assume by passing to a subnet that $\langle z_\alpha\rangle _\alpha$  converges to a point $z\in \overline{V}$. 
Due to  $V=\bigcup_\alpha Y_\alpha$ and $z_\alpha\in \partial Y_\alpha$, the point $z$ has to ly in  $\partial V\subset \supp(\nu)$. 
In the following, we use this net to get a contradiction to $u_\nu(\cdot ,\mu)\leq M$ on~$\partial V$.
Recall that  $\Gamma$ is a skeleton of $\Xan$ containing $x_0$.

If $z\in \partial V\backslash \Gamma$, there exists an open ball $B_z$ in $\Xan\backslash \Gamma$ containing $z$. 
We can find $B_z$ such that  $\overline{B_z}=B_z\cup \{\zeta_z\}\subset \Xan\backslash \Gamma$. 
We may assume $\langle z_\alpha\rangle_\alpha$ to ly in $B_z$.
Then every path from a $z_\alpha$  to $x_0$, or more generally to the skeleton, goes by construction through $\zeta_z$. 
Hence for  every $\alpha$ the path $[z_\alpha,\zeta_z]$ lies inside $Y_\alpha$ as $z_\alpha$ and $x_0$ do, and so $u_\nu(z_\alpha,\mu)\geq u_\nu(\cdot,\mu)$ on~$[z_\alpha,\zeta_z]$ by (\ref{STRU}).
Assume we have equality for every $\alpha$, then $$u_\nu(\cdot,\mu)\equiv u_\nu(z_\alpha,\mu)\geq u_\nu(x_0,\mu)$$  on~$(z,\zeta_z]$ since we can write $(z,\zeta_z]\subset \bigcup_\alpha [z_\alpha,\zeta_z]$ as $z_\alpha$ converges to $z$.
 Proposition~\ref{U} implies
$$u_\nu(z,\mu)=\lim_{x\in[\zeta_z,z),~x\to z}u_\nu(x,\mu)=u_\nu(\zeta_z,\mu)\geq u_\nu(x_0,\mu)>M$$ contradicting $u_\nu(\cdot,\mu)\leq M$ on~$\supp(\nu)$.
Consequently, we may assume that there is a $z_\alpha $  and a point $y_\alpha\in (z_\alpha,\zeta_z]$ such that  $u_\nu(z_\alpha,\mu)> u_\nu(y_\alpha,\mu)$.
Our function $u_\nu(\cdot,\mu)$ is subharmonic on the connected open subset $V\cap B_z$ and $z_\alpha\in I(V\cap B_z)$, and so there exists a path $\Lambda$ from $z_\alpha$ to a boundary point of $V\cap B_z$  by Lemma~\ref{path} such that $u_\nu(\cdot,\mu)$ is non-decreasing along  $\Lambda$.
The boundary points of $V\cap B_z$ consist of points in $\partial V$ and $\zeta_z$.  
Since we have already seen that there is a point $y_\alpha\in (z_\alpha,\zeta_z]$ such that  $u_\nu(z_\alpha,\mu)> u_\nu(y_\alpha,\mu)$, $\Lambda$ cannot be the path $[z_\alpha,\zeta_z]$.
Hence $\Lambda$ is a path to a boundary  point $z'\in\partial V$ and  we get the contradiction
$$u_\nu(z',\mu)=\lim_{x\in\Lambda^\circ,~x\to z'}u_\nu(x,\mu)\geq u_\nu(z_\alpha,\mu)>u_\nu(x_0,\mu)> M,$$
where $u_\nu(\cdot,\mu)$ restricted to $\Lambda$ is continuous by Proposition~\ref{U}.

If $z\in \partial V\cap \Gamma$, we show that $\langle\tau_\Gamma(z_\alpha)\rangle_\alpha$ defines a net in $V\cap \Gamma$ converging to $z$ with $u_\nu(\tau_\Gamma(z_\alpha),\mu)\geq u_\nu(x_0,\mu)> M$ for every $\alpha$. 
Then we use again Proposition~\ref{U}.
Since $\tau_\Gamma$ is continuous, the net $\langle\tau_\Gamma(z_\alpha)\rangle_\alpha$ converges to $\tau_\Gamma(z)=z$.
Clearly, $\langle\tau_\Gamma(z_\alpha)\rangle_\alpha$ lies in $\Gamma$.
The open set $V$ is path-connected, and so there exists a path between $z_\alpha$ and $x_0$  in $V$.
By the construction of the retraction map and due to $x_0\in \Gamma$, $\tau_\Gamma(z_\alpha)$ lies inside this path, and hence it lies in $V$.
We continue with $u_\nu(\tau_\Gamma(z_\alpha),\mu)\geq u_\nu(x_0,\mu)$ for every $z_\alpha$.
Assume that $z_\alpha\neq \tau_\Gamma(z_\alpha)$ because otherwise we are done by (\ref{STRU}).
Denote by $B_\alpha$ the connected component of $\Xan\backslash \Gamma$ containing $z_\alpha$, and choose a sequence of type II points $\zeta_n\in [z_\alpha,\tau_\Gamma(z_\alpha)]^\circ$ converging to $\tau_\Gamma(z_\alpha)$. 
Note that there is  only one path from $z_\alpha$ to $\tau_\Gamma(z_\alpha)$ in $\Xan$, and this path lies in $V$
because $z_\alpha,\tau_\Gamma(z_\alpha)\in V$ and $V$ is path-connected. 
Thus each $\zeta_n$ lies in $V$ as well. 
Let $B_{\alpha,n}$ be the open ball containing $z_\alpha$ and having $\zeta_n$ as unique boundary point.
Since $u_\nu(\cdot,\mu)$ is subharmonic on~$V\cap B_{\alpha,n} $ for every $n\in \N$,   there is a path $\Lambda_n$ from $z_\alpha$ to a boundary point $z_n'$ in $\partial (V\cap B_{\alpha,n})\subset \partial V \cup \{\zeta_n\} $ such that $u_\nu(\cdot,\mu)$ is non-decreasing along  $\Lambda_n$ by Lemma~\ref{path}.
If there exists an $n\in \N$ with $z_n'\in \partial V$, then Proposition~\ref{U} and (\ref{STRU}) imply  
$$u_\nu(z_n',\mu)=\lim_{x\in\Lambda_n^\circ,~x\to z_n'}u_\nu(x,\mu)\geq u_\nu(z_\alpha,\mu)\geq u_\nu(x_0,\mu)> M$$
contradicting $u_\nu(\cdot,\mu)\leq M$ on~$\supp(\nu)$.
Hence  $\Lambda_n=[z_\alpha,\zeta_n]$ for all $n\in \N$.
Recall that $(\zeta_n)_n$ is a sequence in $V$ converging to $\tau_\Gamma(z_\alpha)\in V$. 
Since $u_\nu(\cdot,\mu)$ is continuous on~$V$ and $u_\nu(\cdot,\mu)$ is non-decreasing along $\Lambda_n=[z_\alpha,\zeta_n]$, Proposition~\ref{U} yields 
\begin{align}\label{STRU2}
u_\nu(\tau_\Gamma(z_\alpha),\mu)=\lim_{n\to \infty}u_\nu(\zeta_n,\mu)\geq u_\nu(z_\alpha,\mu).
\end{align}
Altogether, we have a net $(\tau_\Gamma(z_\alpha))_\alpha$  in $V\cap \Gamma$ converging to $z$ such that $$u_\nu(\tau_\Gamma(z_\alpha),\mu)\geq u_\nu(x_0,\mu)> M$$ for every $\alpha$. 
Proposition~\ref{U} tells us that $u_\nu(\cdot,\mu)$ restricted to $\Gamma$ is continuous, and hence using (\ref{STRU}) and (\ref{STRU2}) we get
$$u_\nu(z,\mu)=\lim_{\alpha}u_\nu(\tau_\Gamma(z_\alpha),\mu)\geq u_\nu(x_0,\mu)> M$$
contradicting $u_\nu(\cdot,\mu)\leq M$ on~$\supp(\nu)$.

Hence there cannot exist a point $x_0$ in $V$ with $u_\nu(x_0,\mu)> M$.
\end{proof}

\begin{defn}
	Let $\mu$ be a probability measure with continuous potentials, then we define the \emph{$\mu$-Robin constant} as 
	$$V(\mu):=\inf_{\nu} I_\mu(\nu),$$ where $\nu$ is running over all probability measures supported on~$\Xan$.
\end{defn}
\begin{lem}\label{Vmu} We have $V(\mu)\in \R_{\leq 0}$ and there exists a probability measure $\omega$ on~$\Xan$ such that $I_\mu(\omega)=V(\mu)$.
\end{lem}
\begin{proof}
First, we explain why $V(\mu)$ is a non-positive real number.
	The normalized Arakelov--Green's function $g_\mu$ is  bounded from below as a lsc function on the compact space $\Xan\times \Xan$ by Proposition~\ref{G}, and hence we have $$V(\mu)=\int\int g_\mu(x,y)~d\nu(x)d\nu(y)>-\infty.$$
	On the other hand, $$V(\mu)\leq I_\mu(\mu)=\int \int g_\mu(x,y)~d\mu(x)d\mu(y)=0$$  by the normalization of $g_\mu$.
Thus  $V(\mu)\in \R_{\leq 0}$.

We show the  second part of the assertion applying the same argument used to prove the existence of an equilibrium measure in \cite[Proposition 6.6]{BR}.
Let $\omega_i$ be a sequence of probability measures such that $\lim_{i\to\infty}I_\mu(\omega_i)=V(\mu)$.
By Proposition~\ref{Proho}, we can pass to a subsequence converging weakly to a probability measure $\omega$ on~$\Xan$.
Due to $I_\mu(\omega)\geq V(\mu)$ by the definition of the Robin constant, it remains to show the inequality  $I_\mu(\omega)\leq V(\mu)$.
By Proposition~\ref{G}, the normalized Arakelov--Green's function $g_\mu$ is lsc on the compact space $\Xan\times\Xan$, and so it is bounded from below  by some constant $M\in \R$. 
Proposition~\ref{Radon} tells us that $\omega$ is a Radon measure, and so \cite[Proposition A.3]{BR} yields the following description
$$I_\mu(\omega)=\int\int g_\mu(x,y)~d\omega(x)d\omega(y)=\sup_{\substack{g\in \CS^0(\Xan\times\Xan),\\ M\leq g\leq g_\mu}}\int \int g(x,y)~d\omega(x)d\omega(y),$$  for the space  $\CS^0(\Xan\times\Xan)$ of real-valued continuous functions on~$\Xan\times\Xan$.
For every $g\in \CS^0(\Xan\times\Xan)$ satisfying  $M\leq g\leq g_\mu$, we have 
\begin{align*}
\int \int g(x,y)~d\omega(x)d\omega(y)&=\lim_{i\to\infty}\int \int g(x,y)~d\omega_i(x)d\omega_i(y)\\
& \leq \lim_{i\to\infty}\int \int g_\mu(x,y)~d\omega_i(x)d\omega_i(y)\\
& = \lim_{i\to\infty} I_\mu(\omega_i)=V(\mu),
\end{align*} where the first identity is proven for example in \cite[Lemma 6.5]{BR} and the inequality holds as  every $\omega_i$ is  positive.
Hence $I_\mu(\omega)\leq V(\mu)$.
\end{proof}
\begin{Thm}[Frostman]\label{Frostman}
	Let $\mu$ be a probability measure on~$\Xan$ with continuous potentials and let $\omega$ be a probability measure on~$\Xan$ such that $I_\mu(\omega)=V(\mu)$. Then we have on $\Xan$
	$$u_\omega(\cdot,\mu)\equiv V(\mu).$$
\end{Thm}
\begin{proof}The strategy is as in the proof of \cite[Proposition 8.55]{BR} with using analogous capacity results from Section~\ref{Capacity}.\\[0.2cm]
	
	\textbf{1. Step: } Show that $E:=\{x\in \Xan\mid u_\omega(x,\mu)< V(\mu)\}\subset X(K)$.\\[0.05cm]
          
           By Lemma~\ref{CapII}, it remains to show that $E$ is a proper subset of $\Xan$ of capacity zero.
           Assume that $E\subset \supp(\omega)$, then we get the contradiction
    $$V(\mu)= I_\mu(\omega)=\int \int  g_{\mu}(x,y) ~d\omega (y)d\omega(x)=\int u_\omega(x,\mu)~d\omega(x) < V(\mu).$$
Thus  there has to be a point $\xi \in \supp(\omega)\backslash E$, and so $E$ is indeed a proper subset of $\Xan$. To show that it has capacity zero,
	we consider
	\begin{align*}
	E_n&:= \{x\in \Xan\mid u_\omega(x,\mu)\leq V(\mu)-1/n\}
	\end{align*} for every $n\in  \N_{\geq 1}$.
Clearly, $\xi\notin E_n$ for every $n\in \N_{\geq 1}$.
	Since $ u_\omega(\cdot,\mu)$ is lsc on~$\Xan$ by Proposition~\ref{U}, each $E_n$ is closed and so  compact as a closed subset of a compact space. 
If every $E_n$ has capacity zero, then $E=\bigcup_{n\in \N_{\geq 1}} E_n$ has capacity zero as well by
Corollary~\ref{CC}.

 We therefore assume that there is an $E_n$ with positive capacity, i.e.~there exist  a probability measure $\nu$ supported on~$E_n$, a base point $\zeta_0\in I(\Xan)$ and $\zeta\in \Xan\backslash E_n$ such that $ I_{\zeta_0,\zeta}(\nu)<\infty$.
Since $E_n$ is closed and $I(\Xan)$ is a dense subset of $\Xan$, we may choose $\zeta_0=\zeta\in  I(\Xan)\backslash E_n$ by  Remark~\ref{CapP}. Then
\begin{align}\label{doppelt}
I_{\zeta_0,\zeta_0}(\nu)=\int\int g_{\zeta_0}(\zeta_0,x,y)~d\nu(x)d\nu(y)=\int\int g_{\zeta_0}(x,y)~d\nu(x)d\nu(y)<\infty,\end{align}
where  we used $g_{\zeta_0}(\zeta_0,x,y)=g_{\zeta_0}(x,y)$ from Corollary~\ref{3Var}.
	We can write by the definition of the Arakelov--Green's function $ g_{\mu}$
	\begin{align*}
	I_\mu(\nu) =  I_{\zeta_0,\zeta_0}(\nu)-2 \int\int g_{\zeta_0}(x,\zeta)~d\mu(\zeta)d\nu(x)+C_{\zeta_0}.
	\end{align*}
Since $\mu$ has continuous potentials, the term $2 \int\int g_{\zeta_0}(x,\zeta)~d\mu(\zeta)d\nu(x)$ is finite. Hence 
 $I_{\zeta_0,\zeta_0}(\nu)<\infty$ implies $I_\mu(\nu) < \infty$. 

	Recall that $\xi$ is a point in $\supp(\omega)\backslash E_n$ and $u_\omega(\xi,\mu)\geq V(\mu)$. 
Since  $u_\omega(\cdot,\mu)$ is lsc on~$\Xan$ by Proposition~\ref{U}, we can find an open neighborhood $U$ of $\xi$ such that $u_\omega(\cdot,\mu) > V(\mu)-1/(2n)$ on~$\overline{U}$. Then $\overline{U}\cap E_n=\emptyset$ and $M:=\omega(\overline{U})> 0$ using that $\omega$ is a positive measure and $\xi\in\overline{U}\cap \supp(\omega)$.  We define the following measure on~$\Xan$
	\begin{align*}
	\sigma := \begin{cases}
	M\cdot \nu & \text{ on } E_n,\\
	-\omega & \text{ on } \overline{U},\\
	0 &\text{ elsewhere.}
	\end{cases}
	\end{align*}
	Then $\sigma(\Xan)=M\cdot \nu(E_n)-\omega(\overline{U})=0$ as $\nu$ is a probability measure supported on~$ E_n$. 
Moreover, we can consider
	\begin{align*}
	I_\mu(\sigma) & := \int \int g_\mu(x,y)~d\sigma(x)d\sigma(y) \\
	&~ = M^2\cdot \int_{E_n} \int_{E_n} g_\mu(x,y)~d\nu(x)d\nu(y) - 2M\cdot \int_{E_n}\int_{\overline{U}} g_\mu(x,y)~d\nu(x)d\omega(y)\\
	& ~~~  + \int_{\overline{U}}\int_{\overline{U}} g_\mu(x,y)~d\omega(x)d\omega(y).
	\end{align*} 
We will explain why $I_\mu(\sigma)$ is finite. Note that $g_\mu$ is lsc on the compact space $\Xan\times \Xan$ (cf.~Proposition~\ref{G}), and so bounded from below.
The first term is equal to $M^2 \cdot I_\mu(\nu)$, and we have already seen that $I_\mu(\nu)<\infty$. 
Since $g_\mu$ is bounded from below and $\nu$ is a positive measure, the first term is finite. 
	The second term is finite because $\overline{U}$ and $E_n$ are compact disjoint sets and $g_\mu$ is continuous off the diagonal (see Proposition~\ref{G}).
	The third term has to be finite as well as $g_\mu$ is bounded from below, $\omega$ is a positive measure, and we have  $$\int\int g_\mu(x,y)~d\omega(x)d\omega(y)=I_\mu(\omega)=V(\mu)\in \R$$ by Lemma~\ref{Vmu}.
Consequently, $I_\mu(\sigma)$ is finite.

	For every $t\in [0,1]$, we define the probability measure $\omega_t:=\omega +t\sigma$ on~$\Xan$. 
	Then 
	\begin{align*}
	I_\mu(\omega_t)- I_\mu(\omega) & = \int\int g_\mu(x,y)~d\omega_t(x)d\omega_t(y) - \int\int g_\mu(x,y)~d\omega(x)d\omega(y)  \\& = \int\int g_\mu(x,y)~d\omega(x)d\omega(y) + 2\int\int g_\mu(x,y)~d\omega(x)d(t\sigma)(y)  \\
&~~~+ \int\int g_\mu(x,y)~d(t\sigma)(x)d(t\sigma)(y) - \int\int g_\mu(x,y)~d\omega(x)d\omega(y) \\
&= 2t\cdot \int u_\omega(y,\mu)~d\sigma(y)+t^2 \cdot I_\mu(\sigma).
	\end{align*} Inserting the definition of the measure $\sigma$, we obtain
\begin{align*}
	I_\mu(\omega_t)- I_\mu(\omega)  = 2t\cdot \left( M\cdot \int_{E_n} u_\omega(y,\mu)~d\nu(y) - \int_{\overline{U}} u_\omega(y,\mu)~d\omega(y) \right)+t^2 \cdot I_\mu(\sigma).
	\end{align*} 
Since $u_\omega(\cdot,\mu) \leq V(\mu)-1/n$  on~$E_n$ and $\supp(\nu)\subset E_n$, $u_\omega(\cdot,\mu) > V(\mu)-1/(2n)$ on~$\overline{U}$ and $M=\omega(\overline{U})>0$, we get
	\begin{align*}
	I_\mu(\omega_t)- I_\mu(\omega)
	& \leq  2t\cdot ( M\cdot( V(\mu)-1/n)- M\cdot(V(\mu)-1/(2n)))+t^2 \cdot I_\mu(\sigma)\\
	& = (-M/n)\cdot t +t^2 \cdot I_\mu(\sigma).
	\end{align*} 
	The right hand  side is negative for sufficiently small $t>0$ as $I_\mu(\sigma)$ is finite, and so this contradicts $I_\mu(\omega)=V(\mu)$.
	Hence each $E_n$ has capacity zero, and so does $E$.
	By Lemma~\ref{CapII}, we get $E\cap \HM(\Xan)=\emptyset$.\\[0.3cm]
	
	\textbf{2. Step: } Show that $\omega(E)=0$. \\[0.05cm]	

Pick a base point  $\zeta_0\in I(\Xan)$.
We have seen in Step 1 that $E\subset X(K)$, so $\zeta_0$ cannot by contained in $E$.
Because of $ I_{\zeta_0,\zeta_0}(\omega)=\int \int g_{\zeta_0}(x,y) ~d\omega(x)d\omega(y)$ by Corollary~\ref{3Var}, we have
	\begin{align*}
	I_\mu(\omega) & = \int \int  g_{\mu}(x,y) ~d\omega(y)d\omega(x)\\
	& =\int \int g_{\zeta_0}(x,y) ~d\omega(y)d\omega(x)- \int \int g_{\zeta_0}(x,\zeta)~d\mu(\zeta) d\omega(x)\\
	& ~~~-\int \int g_{\zeta_0}(y,\zeta)~d\mu(\zeta) d\omega(y)+C_{\zeta_0}\\
	& = I_{\zeta_0,\zeta_0}(\omega)-2 \int\int g_{\zeta_0}(x,\zeta)~d\mu(\zeta)d\omega(x)+C_{\zeta_0},
	\end{align*} where the double integral is finite since $\mu$ has continuous potentials.
As $I_\mu(\omega)=V(\mu)$ is finite by Lemma~\ref{Vmu}, it follows  directly from the calculation that $I_{\zeta_0,\zeta_0}(\omega)< \infty$.
Moreover, we have seen in the proof of Step 1 that $E$ has capacity zero and we also know that $\zeta_0\notin E$.  Lemma~\ref{LC} yields $\omega(E)=0$.\\[0.2cm]

	\textbf{3. Step: } Show that $u_\omega(\cdot,\mu)\leq V(\mu)$ on~$\Xan$. \\[0.05cm]
	
Using Maria's theorem~\ref{Maria}, it remains to prove $u_\omega(\cdot,\mu)\leq V(\mu)$ on~$\supp(\omega)$. 
	Assume there is a point $z\in \supp(\omega)$ such that  $u_\omega(z,\mu)> V(\mu)$. Choose $\varepsilon>0$ such that
	$u_\omega(z,\mu)> V(\mu)+\varepsilon.$ 
	Since $u_\omega(\cdot,\mu)$ is lsc on~$\Xan$ by Proposition~\ref{U}, there is an open neighborhood $U_z$ of $z$ with $u_\omega(\cdot,\mu)> V(\mu)+\varepsilon$ on~$U_z$.
	Then $\omega(U_z)>0$ as $z\in\supp(\omega)$.
	By the construction of $E$, we have $u_\omega(\cdot,\mu)< V(\mu)$ on~$E$.
	Hence $E$ and $U_z$ are disjoint and we get the following decomposition of $V(\mu)=I_\mu(\omega)$ 
\begin{align*}
	V(\mu)&= \int_{\Xan}  u_\omega(x,\mu)~d\omega(x)\\
	& =  \int_{U_z}  u_\omega(x,\mu)~d\omega(x) + \int_{\Xan\backslash (U_z\cup E)}  u_\omega(x,\mu)~d\omega(x).
	\end{align*}
	Note that we also use that  the integral of $u_\omega(\cdot,\mu)$ over $E$ has to be zero as $\omega(E)=0$ by Step 2.
For the first term we know that  $u_\omega(\cdot,\mu)> V(\mu)+\varepsilon$ on~$U_z$ and $\omega(U_z)>0$.
Thus 
\begin{align}\label{1}
\int_{\Xan}  u_\omega(x,\mu)~d\omega(x) \geq \omega(U_z)\cdot (V(\mu)+\varepsilon).\end{align}
We have $u_\omega(\cdot,\mu)\geq V(\mu)$ on~$\Xan\backslash E$ by the definition of $E$, and so 
\begin{align}\label{2}\int_{\Xan\backslash (U_z\cup E)}  u_\omega(x,\mu)~d\omega(x) \geq (1-\omega(U_z)-\omega(E))\cdot V(\mu).\end{align}
Putting  (\ref{1}), (\ref{2}) and $\omega(E)=0$ together, 
we get the contradiction
	\begin{align*}
	V(\mu)&\geq \omega(U_z)\cdot (V(\mu)+\varepsilon) + (1-\omega(U_z)-\omega(E))\cdot V(\mu) \\
	& =\omega(U_z)\cdot (V(\mu)+\varepsilon) + (1-\omega(U_z))\cdot V(\mu)\\
	& = V(\mu)+ \omega(U_z)\varepsilon > V(\mu).
	\end{align*}
	Hence $u_\omega(\cdot,\mu)\leq V(\mu)$ on~$\supp(\omega)$.
		Maria's theorem~\ref{Maria},  implies that $u_\omega(\cdot,\mu)\leq V(\mu)$ on~$\Xan$.
		This shows the third step.
	\\[0.1cm]

	By the first step we know that $u_\omega(\cdot,\mu)\geq V(\mu)$ on~$\Xan\backslash X(K)$.
	For every point $y\in X(K)$, we can find  a path $[z,y]$ from a point $z\in I(\Xan)$ to $y$ such that $[z,y)$  is contained in $I(\Xan)\subset \Xan\backslash X(K)$. 
	Then Proposition~\ref{U} implies $$u_\omega(y,\mu)=\lim_{x\in [z,y)}u_\omega(x,\mu)\geq V(\mu).$$
	Hence $E=\{x\in \Xan\mid u_\omega(x,\mu)< V(\mu)\}$ is empty, and so  $u_\omega(\cdot,\mu)\geq V(\mu)$ on~$\Xan$. Step~3 implies $u_\omega(\cdot,\mu)\equiv V(\mu)$ on~$\Xan$.
\end{proof}

\begin{proof}[Proof of Theorem~\ref{THM}]
	Let $\omega$ be a probability measure on~$\Xan$ that minimizes the energy integral, i.e.~$I_\mu(\omega)=V(\mu)$.
	Such a measure always exists by Lemma~\ref{Vmu}. 
	By Frostman's theorem~\ref{Frostman}, $u_\omega(\cdot,\mu)$ is constant on $\Xan$, and hence 
	$$0=dd^c u_\omega(\cdot,\mu) = \mu-\omega$$ by Proposition~\ref{U}.
	Thus $\omega$ minimizes the energy integral if and only if $\omega =\mu$.
	Since $I_\mu(\mu)=\int \int  g_{\mu}(x,y) ~d\mu (y)d\mu(x)=0$ by the normalization of the Arakelov--Green's function $g_\mu$, it follows that $I_\mu(\nu)\geq 0$ for every probability measure $\nu$ on~$\Xan$.
\end{proof}

\begin{kor}
Let $\zeta\in I(\Xan)$ and $\mu$ be a probability measure on~$\Xan$ with continuous potentials. Then $g_{\mu}(\zeta,\zeta)\geq 0$, and $g_{\mu}(\zeta,\zeta)=0$ if and only if $\mu=\delta_{\zeta}$.
\end{kor}
\begin{proof}
Since 
$$g_{\mu}(\zeta,\zeta)=\int\int g_\mu(x,y)~d\delta_{\zeta}(x)d\delta_{\zeta}(y)= I_{\mu}(\delta_{\zeta}),$$ the Energy Minimization Principle (Theorem~\ref{THM}) gives the assertion immediately.
\end{proof}

\section{Local discrepancy}\label{SubLD}
Let $E$ be an elliptic curve over $K$ with $j$-invariant $j_E$. 
In this section, we give a different proof of the local discrepancy result from \cite[Corollary 5.6]{BP} using our Energy Minimization Principle (Theorem~\ref{THM}).

\begin{bem}
	In the following, let $\Gamma_E$ be the minimal skeleton of $E^{\an}$. 
	Then $\Gamma_E$ is a single point $\zeta_0$ when $E$ has good reduction and $\Gamma_E$ corresponds to the circle $\R /\Z$ when it has multiplicative reduction.
	One has a canonical probability measure $\mu_E$ supported on~$\Gamma_E$, where 
	\begin{enumerate}
		\item $\mu_E$ is the dirac measure in $\zeta_0$ if $E$ has good reduction, and
		\item $\mu_E$ is the uniform probability measure (i.e.~Haar measure) supported on the circle $\Gamma_E\simeq \R /\Z$ if $E$ has multiplicative reduction.
	\end{enumerate}
Then $\mu_E$ has in particular continuous potentials  by Example~\ref{log-c}. 
Hence we can consider its corresponding Arakelov--Green's function $g_{\mu_E}$ on  $E^{\an}\times E^{\an}$.
\end{bem}

\begin{defn}Let $Z=\{P_1,\ldots,P_N\}$ be a set of $N$ distinct points in $E(K)$.
Then the \emph{local discrepancy} of $Z$ is defined as 
$$D(Z):=\frac{1}{N^2}\left( \sum_{i\neq j}g_{\mu_E}(P_i,P_j) + \frac{N}{12}\log ^+|j_E|\right).$$
\end{defn}

\begin{bem}Baker and Petsche defined  the local discrepancy in \cite[\S 3.4]{BP} and \cite[\S 2.2]{Pe} of a set $Z=\{P_1,\ldots,P_N\}$ of $N$ distinct points in $E(K)$ as
	$$\frac{1}{N^2}\left( \sum_{i\neq j}\lambda(P_i-P_j) + \frac{N}{12}\log ^+|j_E|_v\right)$$ for the N\'eron function $\lambda\colon E(K)\backslash \{O\}\to \R$ (cf.~\cite[\S VI.1]{Sil}).	
	
Note that our definition is consistent with theirs. As it is also mentioned in \cite[Remark 5.3]{BP}, the N\'eron function can be extend to an Arakelov--Green's function corresponding to the canonical measure $\mu$ on  $E^{\an}$. By the uniqueness of the Arakelov--Green's function (see Remark~\ref{AGunique}), we have $g_{\mu_E}(P,Q)=\lambda(P-Q)$ for  $P\neq Q\in E(K)$.
\end{bem}

Baker and Petsche showed in \cite[Corollary 5.6]{BP} the following result for the local discrepancy when $K=\C_v$. 
Here,~$v$ is a non-archimedean place of a number field $k$ and $\C_v$ is the completion of the algebraic closure of the completion of $k$ with respect to $v$.
We can prove this statement for our general $K$ using our characterization of the local discrepancy and the Energy Minimization Principle (Theorem~\ref{THM}).

\begin{kor}\label{KorLD}
	For each $n\in \N$, let $Z_n\subset E(K)$  be a set consisting of $n$ distinct points and let $\delta_n$ be the probability measure on~$E^{\an}$ that is equidistributed on~$Z_n$. If $\lim_{n\to \infty} D(Z_n)=0$, then $\delta_n$ converges weakly to $\mu_E$ on~$E^{\an}$.
\end{kor} 
\begin{proof}By passing to a subsequence we may assume that $\delta_n$ converges weakly to a probability measure $\nu$ on~$E^{\an}$ (see Proposition~\ref{Proho}).	
We show that $I_{\mu_E}(\nu)$ is zero and we then use the Energy Minimization Principle~\ref{THM}.
We have seen  in the Energy Minimization Principle that $I_{\mu_E}(\nu)\geq 0$.
Thus it remains to show $I_{\mu_E}(\nu)\leq 0$.
Due to the definition of the $\mu_E$-energy integral and  \cite[Lemma 7.54]{BR}, the following inequality holds
\begin{align*}
	I_{\mu_E}(\nu)&=\int \int_{E^{\an}\times E^{\an}} g_{\mu_E}(x,y)~d\nu(x)d\nu(y)\\
	&\leq \liminf_{n\to \infty} \int \int_{(E^{\an}\times E^{\an})\backslash \Delta } g_{\mu_E}(x,y)~d\delta_n(x)d\delta_n(y)\\
	&=\liminf_{n\to \infty} \frac{1}{n^2}\sum_{P\neq Q\in Z_n}g_{\mu_E}(P,Q),
\end{align*}	
	where $\Delta:=\mathrm{Diag}(E^{\an})$.
Since $D(Z_n)=\frac{1}{n^2}\sum_{P\neq Q\in Z_n}g_{\mu_E}(P,Q) + \frac{1}{12n}\log ^+|j_E|$ converges to zero, and $\frac{1}{12n}\log ^+|j_E|$ does as well,  we have \begin{align*}
\lim_{n\to \infty} \frac{1}{n^2}\sum_{P\neq Q\in Z_n}g_{\mu_E}(P,Q) =0.
\end{align*}
Hence $I_{\mu_E}(\nu)\leq 0$. 
The Energy Minimization Principle yields $\mu_E=\nu$.
\end{proof}

\bibliographystyle{alpha}
\def\cprime{$'$}

\end{document}